\documentclass[aap]{imsart}

\pdfoutput=1

\usepackage[linesnumbered,ruled]{algorithm2e}
\RestyleAlgo{ruled}
\usepackage{xcolor}
\usepackage{mathrsfs,mathtools,latexsym, enumerate, amsmath, amsthm, amssymb,hyperref,amsfonts, pdftricks,tikz,framed,color,enumerate,xcolor,bbm,appendix,soul}
\usepackage{romanbar}
\usepackage{MnSymbol} 
\usepackage{algorithmic}
\usepackage{multirow,mathtools,latexsym, mathrsfs, pdftricks,framed,listings,mleftright,enumitem}
\RequirePackage[numbers]{natbib}

\usepackage{romanbar}

\startlocaldefs

\newcommand{\norm}[1]{\left\lVert#1\right\rVert}
\newcommand{\R}{\mathbb{R}}

\newcommand{\ev}[1]{\mathbb{E} \left[ #1 \right]}
\newcommand{\evsub}[2]{\mathbb{E}_{#1} \left[ #2 \right]}

\DeclarePairedDelimiter{\abs}{\lvert}{\rvert}
\newcommand{\dd}{\, \text{d}}

\theoremstyle{plain}% default
\newtheorem{prototheorem}{Theorem}

\newtheorem{algorithm3}{Algorithm}

\newtheorem{theorem}[prototheorem]{Theorem}
\newtheorem{lemma}[prototheorem]{Lemma}
\newtheorem{proposition}[prototheorem]{Proposition}
\newtheorem{corollary}[prototheorem]{Corollary}
\newtheorem{assumption}[prototheorem]{Assumption}
\newtheorem{definition}[prototheorem]{Definition}
\newtheorem{remark}[prototheorem]{Remark}
\newtheorem{example}[prototheorem]{Example}

\newcommand{\eps}{\varepsilon}

\newcommand{\Z}{\ensuremath \mathbb{Z}}
\newcommand{\W}{\boldsymbol{\mathcal{W}}}
\newcommand{\PM}{\ensuremath \mathcal{P}}
\newcommand{\J}{\mathrm{Couplings}}
\newcommand{\law}{\operatorname{Law}}

\newcommand{\lb}[1]{{\lfloor #1 \rfloor_h}}
\newcommand{\ub}[1]{{\lceil #1 \rceil_h}}
\newcommand{\rn}[1]{\Romanbar{#1}}
\newtheorem*{algorithm3*}{Algorithm}
\endlocaldefs

\begin{document}

\begin{frontmatter}
\title{Unadjusted Hamiltonian MCMC with \\  Stratified Monte Carlo Time Integration }
\runtitle{uHMC with sMC Time Integration}

\begin{abstract} 
A randomized time integrator is suggested for unadjusted Hamiltonian Monte Carlo (uHMC) which involves a very minor modification to the usual Verlet time  integrator, and hence, is easy to implement. 
For target distributions of the form $\mu(dx) \propto e^{-U(x)} dx$ where $U: \mathbb{R}^d \to \mathbb{R}_{\ge 0}$ is $K$-strongly convex but only $L$-gradient Lipschitz, and initial distributions $\nu$ with finite second moment,  coupling proofs reveal that an $\varepsilon$-accurate approximation of the target distribution
in $L^2$-Wasserstein distance $\boldsymbol{\mathcal{W}}^2$ can be achieved by the uHMC algorithm with randomized time integration using $O\left((d/K)^{1/3} (L/K)^{5/3} \varepsilon^{-2/3} \log( \boldsymbol{\mathcal{W}}^2(\mu, \nu) / \varepsilon)^+\right)$
gradient evaluations; whereas for such rough target densities  the corresponding complexity of the  uHMC algorithm with Verlet time integration is in general $O\left((d/K)^{1/2} (L/K)^2 \varepsilon^{-1} \log( \boldsymbol{\mathcal{W}}^2(\mu, \nu) / \varepsilon)^+ \right)$.  Metropolis-adjustable randomized time integrators are also provided.    %For the latter, we find the complexity of duration-randomized uHMC with sMC time integration improves to  $O\left(\max\left((d/K)^{1/4} (L/K)^{3/2} \varepsilon^{-1/2},(d/K)^{1/3} (L/K)^{4/3} \varepsilon^{-2/3} \right) \right)$ up to logarithmic factors. 
\end{abstract}

\begin{aug}
\author[A]{\fnms{Nawaf} \snm{Bou-Rabee}\ead[label=e1]{nawaf.bourabee@rutgers.edu}},
\and
\author[B]{\fnms{Milo} \snm{Marsden}\ead[label=e2]{mmarsden@stanford.edu}}
\address[A]{Department of Mathematical Sciences \\ Rutgers University Camden \\ 311 N 5th Street \\ Camden, NJ 08102 USA \printead[presep={,\ }]{e1}} 
\address[B]{Department of Mathematics \\ Stanford University \\ 450 Jane Stanford Way \\ Stanford, CA 94025 USA \printead[presep={,\ }]{e2}}
%\runAuthor{N. Bou-Rabee \and M.~Marsden}
\end{aug}


\end{frontmatter}

\maketitle

\section{Introduction}

Hamiltonian Monte Carlo (HMC) is a  gradient-based MCMC method aimed at target probability distributions of the form $\mu(dx) \propto e^{-U(x)} dx$ where $U: \mathbb{R}^d \to \mathbb{R}_{\ge 0}$ is continuously differentiable.
A defining characteristic of HMC is that it incorporates a measure-preserving Hamiltonian flow per transition step \cite{DuKePeRo1987,Ne2011}. This flow is typically approximated using a \emph{deterministic} time integrator, and the discretization bias can either be borne (\emph{unadjusted} HMC or \emph{uHMC} for short) or eliminated by a Metropolis-Hastings filter (\emph{adjusted} HMC).  

In this work, \emph{randomized} time integrators are developed for unadjusted HMC, which improve upon the current state of the art for \emph{rough target densities} \cite{vogrinc2021counterexamples}. That is, we consider the case where $U$ is assumed differentiable with $L$-Lipschitz gradient. But, crucially, we do not assume any regularity beyond this and in particular do not assume a Lipschitz Hessian.

Furthermore, we assume throughout this paper that $U$ is $K$-strongly convex and we state guarantees in $L^2$-Wasserstein distance $\W^2$. The strong convexity assumption can be relaxed to asymptotic strong convexity --- as in \cite{BoEbZi2020,cheng2018sharp, BouRabeeSchuh2023}. However, the resulting contraction rates, and in turn, asymptotic bias and complexity estimates will then depend on model and hyperparameters in a more intricate way. On the other hand, under global strong convexity, the dependence on model and hyperparameters is more transparent,  and therefore, this setting  has been a focus of much of the extant literature \cite{chen2020fast,lee2018algorithmic,monmarche2022united,shen2019randomized} --- as we briefly review below.

\subsection{State of the Art}  At present, the  Verlet time integrator is the method of choice  for both unadjusted and adjusted HMC \cite{BoSaActaN2018}. This is for sound reasons.  Indeed, the Verlet integrator is cheap; like forward Euler it requires only one new gradient evaluation per integration step. The Verlet integrator also has the maximal stability interval for the simple harmonic model problem\footnote{More precisely, among palindromic splitting methods with a fixed computational budget of $N$ evaluations of the potential force per step, the method with the longest stability interval is a concatenation of $N$ Verlet steps \cite[Section 4.5]{BoSaActaN2018}.}\cite{BlCaSa2014,casas2022new}. These properties are quite relevant to uHMC \cite{BouRabeeSchuh2023}. Moreover, the geometric properties of the Verlet integrator (symplecticity and reversibility) are key to making the method Metropolis-adjustable \cite{FaSaSk2014,BoSaActaN2018}.  Unsurprisingly, most research work on HMC has been devoted to studying unadjusted/adjusted HMC with Verlet time integration.

Under the assumption that the potential $U$ has Lipschitz Hessian the Verlet integrator is also second-order accurate.  Absent this assumption -- in the rough potential setting -- the accuracy of the integrator drops to first order. In turn, this drop in accuracy reduces the efficiency of HMC as an MCMC method and it is natural then to ask whether using a different integration scheme can produce a more efficient MCMC method. The work of Lee, Song and Vempala \cite{lee2018algorithmic}, which suggests a  collocation method for the Hamiltonian flow, answers in the affirmative.
This collocation method relies on a choice of basis functions (usually polynomials up to a certain degree) to represent the exact Hamiltonian flow, and uses a nonlinear solver per transition step.  General complexity guarantees are given for these collocation methods, and they specifically consider the special case of a basis of piecewise quadratic polynomials defined on a time grid of step size $h$. In this case, Lee, Song and Vempala prove that uHMC with collocation can in principle produce an $\eps$-accurate approximation of the target  $\mu$ in $\W^2$ distance using \begin{equation} 
\label{collocation_complexity}
O\left( \left(\frac{d}{K} \right)^{1/2} \left( \frac{L}{K} \right)^{7/4} \eps^{-1} \log \left( \frac{d}{K^{1/2}} \frac{L}{K} \frac{1}{\eps} \right)^+    \right) \quad \text{gradient evaluations}
\end{equation}
when initialized at the minimum of $U$ and run with duration $T \propto K^{1/4}/L^{3/4}$ where $\log(\mathsf{a})^+ = \max(\log(\mathsf{a}), 0)$ for $\mathsf{a} > 0$.  The ratio $L/K$ is known as the \emph{condition number} of the function $U$, and it satisfies $L/K \ge 1$.   In each transition step of uHMC, $h$ is chosen to satisfy $h^{-1} \propto LT^3(|v| + |\nabla U(x)| T)K^{1/2}(L/K)^{3/2} \eps^{-1} $, where $x,v \in \mathbb{R}^d$ are respectively the initial position and velocity at the current transition step. See \cite[Theorem 1.6]{lee2018algorithmic} for a detailed statement.  Remarkably, this complexity guarantee requires no further regularity beyond  $K$-strong convexity and  $L$-gradient Lipschitz-ness of $U$.   In practice, because uHMC with collocation requires a nonlinear solve per transition step, its widespread use is currently limited.

%In practice, because the method requires a nonlinear solve per transition step its widespread use is limited. 

%This result improves on the best previously known bound of $O((K/L)^2 (d/K)^{1/2}  \log \left(\left( L/K \right) (d/\eps \sqrt{K}) \right )$ gradient evaluations \cite[Theorem 1]{cheng2018underdamped}. 

For comparison, the corresponding complexity for uHMC with Verlet time integration is \begin{equation} 
\label{verlet_complexity}
O\left( \left(\frac{d}{K} \right)^{1/2} \left( \frac{L}{K} \right)^{2} \eps^{-1} \log \left( \frac{\W^2(\mu,\nu)}{\eps} \right)^+  \right) \quad \text{gradient evaluations} 
\end{equation}
\sloppy when initialized from a distribution $\nu$ with finite second moment and run with time step size $h \propto (L/K)^{-3/2} d^{-1/2} \eps$ and duration $T \propto L^{-1/2}$.  See \cite[Chapter 5]{BouRabeeEberleLectureNotes2020} for detailed statements and proofs of \eqref{verlet_complexity}; which  is based on \cite[Appendix A]{BouRabeeSchuh2023}. %Note that the assumption that the potential be strongly convex and gradient Lipschitz is used to establish the contractivity and bound the asymptotic bias both in Verlet and in the method presented in this paper. These assumptions could be relaxed – as in \cite{BoEbZi2020}. However, this setting is convenient for precise comparison between algorithms, and consequently it has been the focus of much of the existing literature. \cite{shen2019randomized, chen2020fast,lee2018algorithmic}
Note that uHMC with Verlet  underperforms uHMC with collocation  in terms of condition number dependence. This substandard performance is due to the above mentioned drop in accuracy of the Verlet scheme, since uHMC with Verlet then requires a small time step size in order to resolve the asymptotic bias.

In principle, adjusted HMC can filter out all of the asymptotic bias  due to time discretization.  One would therefore hope that the dependence of the complexity on the accuracy parameter $\varepsilon$ is at most logarithmic with adjustment.  This result has recently been demonstrated under higher regularity assumptions and some restrictions on the initial conditions; specifically, assuming $U$ is strongly convex, gradient Lipschitz and Hessian Lipschitz, Chen, Dwivedi, Wainwright, and Yu use a clever conductance argument to prove that adjusted HMC with Verlet from a `warm' start can achieve $\eps$ accuracy in total variation distance using $O(d^{11/12} (L/K) \log(1/\varepsilon)^+)$ gradient evaluations \cite{chen2020fast}. In the same setting, implementable starting distributions are also considered; specifically, from a `feasible' start the dimension and condition number dependence become slightly worse. %$O(\max(d (L/K)^{3/4}, d^{11/12} (L/K), d^{3/4} (L/K)^{5/4}, d^{1/2} (L/K)^{3/2})) \log(1/ \eps)^+)$ gradient evaluations to reach $\eps$ accuracy. 
%Remarkably, in both complexity estimates, the dependence on $\varepsilon$ is logarithmic.
At present,  it remains an open problem to prove such bounds in the rough target density case and to relax the warm/feasible start conditions. %An interesting direction for future research is to incorporate time integration randomization into adjusted HMC and apply a conductance-type argument to obtain a comparable/better complexity guarantee under weaker regularity conditions. 

%Consequently, compared to the $C^3$ case, uHMC with Verlet integration in the absence of higher regularity requires the use much smaller time steps to control the asymptotic bias, which in turn requires more gradient evaluations. To produce an $\eps$-accurate approximation in the case where the potential is assumed only to be $K$-strongly convex and $L$-gradient Lipschitz, uHMC with Verlet integration requires: \begin{equation}   O\left((d/K)^{1/2} (L/K)^2 \varepsilon^{-1} \log( \boldsymbol{\mathcal{W}}^2(\mu, \nu) / \varepsilon)\right) \quad \text{gradient evaluations.} \end{equation}
% Working on this ...will return post-lunch

For analogous MCMC algorithms based on underdamped  Langevin dynamics (ULD), the Randomized Midpoint Method (RMM) of Shen and Lee demonstrates that randomized time integrators can achieve provably optimal performance in the rough target density case \cite{shen2019randomized}.  The RMM method is essentially a randomized time integrator for ULD.  In particular, Shen and Lee prove that  RMM can produce an $\eps$-accurate approximation of the target in $\W^2$ using
\begin{equation} \label{rmm_complexity} O \left(  \left( \frac{d}{K} \right)^{1/3}  \left( \frac{L}{K} \right) ^{7/6}\eps^{-2/3} \left(\log \left(  \frac{d^{1/2}}{K^{1/2} \eps} \right)^+\right)^{4/3}  \right) \quad \text{gradient evaluations}   \end{equation} 
%\begin{equation} \label{rmm_complexity} O \left(  \left( \frac{d}{K} \right)^{1/6}  \left( \frac{L}{K} \right) ^{7/6}\eps^{-1/3} \left(\log \left(  \frac{d^{1/2}}{K^{1/2} \eps} \right)^+\right)^{7/6} + \left( \frac{d}{K} \right)^{1/3} \left(  \frac{L}{K} \right)  \eps^{-2/3} \left( \log \left( \frac{d^{1/2}}{K^{1/2} \eps} \right)^+ \right)^{4/3} \right) ~~ \text{gradient evaluations}   \end{equation} 
when initialized at rest (i.e., with zero initial velocity) within $(d/K)^{1/2}$ of the minimizer of $U$ and with mass $L$, friction parameter $2$, and time step size $h \propto  (K/d)^{1/3} L^{-1/6} \eps^{2/3}$.
%\[ h \propto \min \left( \frac{K^{1/3}}{d^{1/6}L^{1/6}} \eps^{1/3} \left( \log \left( \frac{d^{1/2}}{K^{1/2} \eps} \right)^+ \right)^{1/6},  \frac{K^{1/3}}{d^{1/3}} \eps^{2/3} \left( \log \left( \frac{d^{1/2}}{K^{1/2} \eps} \right)^+ \right)^{1/3} \right), \]
The proof uses a perturbative approach that leverages contractivity of exact ULD \cite{cheng2018underdamped, dalalyan2020sampling}. 
%to bound the $\W^2$ distance between the distribution of the corresponding RMM chain and the target distribution.
Ergodicity of the RMM chain and a $3/2$-order of accuracy for the $\W^2$-asymptotic bias was subsequently proven in \cite{ErgodicityRMMHYB}. Additionally, Cao, Lu, and Wang  demonstrate the optimality of RMM among ULD based MCMC algorithms in the rough target density setting \cite{Cao_2021_IBC}.
% Specifically, the authors show that any randomized algorithm for simulating ULD which makes $N$ queries to $\nabla U$, the driving Brownian motion, and the weighted integral of Brownian motion will suffer worst case $L^2$-error of order at least $\Omega(d^{1/2} N^{-3/2})$. Thus, to guarantee an $\eps$-accurate approximation of ULD in $\W^2$ distance, one requires $N$ to be at least $\Omega( d^{1/3} \eps^{-2/3})$ --- consistent with the upper bound on the $L^2$-error for RMM in dimension $d$ and accuracy $\eps$. 
The paper also surveys the existing literature on information theoretic lower bounds for randomized time integrators for ODEs and SDEs, such as \cite{Kacewicz2004,Kacewicz2005,Daun2011}.

Another related work is the shifted ODE method for ULD due to Foster, Lyons and Oberhauser \cite{FTOshiftedode2021}. In the  strongly convex and rough potential setting considered here, the \emph{exact} shifted ODE method can in principle produce an $\eps$-accurate approximation of the target distribution in $\W^2$ distance using  $O(d^{1/3}\eps^{-2/3}  \log (1/\eps)^+)$ gradient evaluations with better complexity guarantees under higher regularity assumptions.  The shifted ODE method is inspired by rough path theory, in which SDEs are realized as instances of Controlled Differential Equations (CDE). In particular, the shifted ODE method is constructed by tuning a controlling path such that the Taylor expansion of the CDE solution has the same low order terms as ULD.  In practice, the ODE they obtain cannot be exactly solved, and consequently, they propose two discretizations based on a third order Runge-Kutta method and a fourth order splitting method. Numerical results for the discretizations are promising.  The challenge is that the discretizations are trickier to analyze than the exact shifted ODE method.

HMC and ULD are indeed analogous stochastic processes, and in general, one might hope that results for ULD extend to uHMC. Elaborating on this analogy, Monmarch\'{e} considers a parameterized family of unadjusted algorithms that include as special cases uHMC and ULD based algorithms, and permits a simultaneous and unified analysis and comparison \cite{monmarche2022united}. In the strongly convex and rough potential setting, dimension-free lower bounds on the convergence rates of the algorithms are provided. Specializing to the case where $U$ is also Hessian-Lipschitz, and focusing only on the dimension $d$ and accuracy $\eps$ dependence, Monmarch\'{e} finds that both the ULD and uHMC  can produce an $\eps$-accurate approximation of the target in $O(d^{1/2} \eps^{-1/2} \log (d/\eps)^+)$ gradient evaluations.  Moreover, in the Gaussian case,  Monmarch\'{e} notes the rate for ULD and uHMC with partial momentum refreshment improves from $K/L$ to $(K/L)^{1/2}$ dependence.

Considering the strength of the connection between uHMC and ULD highlighted above, and the improved complexity guarantee \eqref{rmm_complexity} provided by the use of randomized integrators for ULD, it is natural to ask: \emph{Is it possible to construct a randomized time integrator for uHMC that confers a better complexity guarantee in the rough target density case?} Since Hamiltonian dynamics does not explicitly incorporate friction or diffusion like ULD does \cite{MaStHi2002,BoSa2017}, it is not obvious that this strategy yields a better complexity guarantee than uHMC with Verlet. At a technical level, understanding the contractivity and asymptotic bias of the uHMC algorithm requires developing new mathematical arguments to resolve the improvement due to time integrator randomization.  This paper answers the above question in the affirmative by suggesting a simple randomized time integrator for uHMC and directly analyzing the contractivity and asymptotic bias of the resulting MCMC algorithm in $\W^2$ distance.

\subsection{Short Summary of Main Results} We now outline our main contributions.  As above, we consider a target distribution $\mu( dx) \propto e^{-U(x)} dx$ where $U: \mathbb{R}^d \to \mathbb{R}_{\ge 0}$ is continuously differentiable, $K$-strongly convex and $L$-gradient Lipschitz. Let  $(q_t(x,v),p_t(x,v))$ denote the exact flow of the Hamiltonian dynamics
\[\frac{d}{dt} {q}_t  \,  =\,     {p}_t \;, \quad  \frac{d}{dt} {p}_t \,  =\, -\nabla U({q}_t) \;, \quad (q_0, v_0) = (x,v) \;. 
\]
The randomized time integrator we suggest for the Hamiltonian flow is a stratified Monte Carlo (sMC) time integrator.  Let $h>0$ be a time step size and $\{ t_k := k h \}_{k \in \mathbb{N}_0}$ be an evenly spaced time grid.  This grid partitions time into subintervals $\{[t_k,t_{k+1})\}_{k \in \mathbb{N}_0}$ termed `strata'.  Let $( \mathcal{U}_{i} )_{i \in \mathbb{N}_0}$ be a sequence of independent random variables such that $ \mathcal{U}_{i} \sim \text{Uniform}(t_i,t_{i+1})$. One step of the sMC time integrator from $t_i$ to $t_{i+1}$ is given by
\begin{equation} \label{intro:schemeeexplicit}
	\Tilde{Q}_{t_{i+1}}  \, = \, \Tilde{Q}_{t_i} + h \tilde{P}_{t_i} + \frac{1}{2} h^2 \Tilde{F}_{t_i} \;, \quad  \Tilde{V}_{t_{i+1}} = \tilde{P}_{t_i} + h \tilde{F}_{t_i} \;,  \quad  (\tilde{Q}_0, \tilde{V}_0) = (x,v) \;,
\end{equation}
where $\Tilde{F}_{t_i}  = -\nabla U (\Tilde{Q}_{t_i} + (\mathcal{U}_i - t_i) \tilde{P}_{t_i})$ is the potential force evaluated at the `random point'  $\Tilde{Q}_{t_i} + (\mathcal{U}_i - t_i) \tilde{P}_{t_i}$.  
%In words, $ \mathcal{U}_{i}$ is a  random temporal sample point  uniformly drawn from the $i$-th stratum  independently of the sample points in the other strata.  
%Note that the sMC time integrator is \emph{explicit} in the sense that $(\Tilde{Q}_{t_{i+1}} , \Tilde{V}_{t_{i+1}})$ is an explicit function of $(\Tilde{Q}_{t_{i}} , \Tilde{V}_{t_{i}})$.

One intuition behind this scheme is as follows. Like Verlet integration, the sMC integrator updates the position variable on the $i$-th stratum $[t_i, t_{i+1})$ by a constant force $\Tilde{F}_{t_i}$. However, unlike Verlet integration, the update rule involves the force evaluated at a random temporal point $\mathcal{U}_i$ sampled from the $i$-th stratum, rather than from the left endpoint $t_i$ of the $i$-th stratum $[t_i, t_{i+1})$.  Moreover, unlike Verlet integration, the sMC integrator updates the velocity variable on the $i$-th stratum $[t_i, t_{i+1})$ by the same constant force $\Tilde{F}_{t_i}$.  Thus, this scheme uses only one new gradient evaluation per sMC integration step.  This scheme is probably the simplest randomized time integrator for the Hamiltonian dynamics, but it certainly is not the only strategy. For example, one could approximate the force over the $i$-th stratum as $-\nabla U(q_{\mathcal{U}_i} (\tilde{Q}_{t_i}, \tilde{P}_{t_i}))$.  However, as the true dynamics are unknown, this is not implementable. Choosing instead to first approximate $q_{\mathcal{U}_i} (\tilde{Q}_{t_i}, \tilde{P}_{t_i})$ using the forward Euler method, and then using the force at the resulting `random point' to approximate the dynamics for both position and velocity over the $i$-th stratum $[t_i, t_{i+1})$ results in the sMC method described above.  Replacing Verlet integration in this way, we obtain the uHMC algorithm with complete momentum refreshment described in Algorithm \ref{intro:alg1}.

%Surprisingly, even though this method uses only one sample point per strata, it has higher order $L^2$-accuracy.

%Incorporating randomness in this way heuristically make the scheme more robust against systematic errors introduced by continually choosing the initial force to approximate the dynamics over $[t_i, t_{i+1})$. 

\begin{algorithm}[h] \label{intro:alg1}
 \SetKwInput{Input}{Input}
 \SetKwInput{Output}{Output}
\Input{Current state of the chain $X_0 \in \mathbb{R}^d$, duration $T>0$ and time step size  $h>0$.}
\caption{uHMC Transition Step with sMC Time Integration}
Sample $\xi$ from $\mathcal{N}(0, I_d)$; \\
Initialize sMC time integrator $\tilde{Q}_0 = X_0$; $\tilde{V}_0=\xi$; $t_0=0$; $n = \lfloor T/h \rfloor$;   \\
\For{$i=0$ \KwTo $i = n-1$} {
Update time $t_{i+1} = t_i + h$; \\
Sample $\mathcal{U}_{i}$ uniformly from $(t_i, t_{i+1})$; \\
Evaluate potential force $\Tilde{F}_{t_i} = -\nabla U (\Tilde{Q}_{t_i} + (\mathcal{U}_i - t_i) \tilde{P}_{t_i})$; \\
Update position $\Tilde{Q}_{t_{i+1}} \, = \,  \Tilde{Q}_{t_i} + h \tilde{P}_{t_i} + \frac{1}{2} h^2 \Tilde{F}_{t_i}$; \\
Update velocity $\Tilde{V}_{t_{i+1}} \, = \, \tilde{P}_{t_i} + h \tilde{F}_{t_i}$; 
}
\Output{Next state of the chain $X_1= \tilde{Q}_{t_n}$.}
\end{algorithm}

The main result of this paper states that uHMC with sMC time integration can produce an $\eps$-accurate approximation of the target distribution using
\begin{equation} \label{smc_complexity}
O\left( \left( \frac{d}{K} \right)^{1/3} \left( \frac{L}{K} \right)^{5/3} \eps^{-2/3} \log \left( \frac{\W^2(\mu,\nu)}{\eps} \right)^+ \right) \quad \text{gradient evaluations}
\end{equation} 
when initialized from an arbitrary distribution $\nu$ with finite second moment and run with the  hyperparameters specified below in \eqref{eq:tuning}. The proof of this complexity guarantee follows from two theorems, which we briefly outline.

Let $\tilde{\pi}$ denote the one-step transition kernel of  uHMC with sMC time integration.  
Firstly, assuming that $L T^2 \le 1/8$ and $h \le T$, Theorem \ref{thm:contractivity} uses a synchronous coupling to prove $\W^2$-contractivity of $\tilde{\pi}$ 
\begin{equation} \label{intro:w2_contr}
    \W^2(\nu \tilde{\pi}, \eta \tilde{\pi}) \, \le \, e^{- c} \, \W^2(\nu, \eta) \quad \text{with $c = K T^2 /6$} \;, \end{equation}
where $\nu, \eta$ are arbitrary probability measures on $\mathbb{R}^d$ with finite second moment.  The $\W^2$-contraction coefficient $e^{-c}$ is uniform in the time step size.  The proof of Theorem \ref{thm:contractivity} relies on Lemma~\ref{lem:contr_SMC}, which states almost sure contractivity of two copies of the sMC time integrator starting with the same initial velocities and with synchronized random temporal sample points. The proof of Lemma~\ref{lem:contr_SMC} crucially relies on $K$-strong convexity of $U$ and the co-coercivity property of $\nabla U$; see Remark~\ref{rmk:cocoercivity} for details on the latter. 
The proof involves a careful balance of these competing effects at the random points where the force is evaluated.
%We emphasize that an analogous $\W^2$-contractivity result can be proven for uHMC with Verlet time integration without assuming higher regularity \cite[Chapter 5]{BouRabeeEberleLectureNotes2020}.  
As a corollary,  $\tilde{\pi}$ admits a unique invariant measure $\tilde{\mu}$, but in general, due to time discretization error $\tilde{\mu} \ne \mu$.  

Secondly, we upper bound the $\W^2$-asymptotic bias of $\tilde{\pi}$, which quantifies  $\W^2(\mu, \tilde{\mu})$.  To this end, let $\pi$ denote the transition kernel of exact HMC, which uses the exact Hamiltonian flow per transition step and satisfies $\mu \pi = \mu$.  Theorem \ref{thm:asymptotic_bias}  uses a synchronous coupling of $\tilde{\pi}$ and $\pi$ to prove: for $LT^2 \le 1/8$ and $h \le T$,
\begin{equation} \label{intro:w2_asymp_bias}
\W^2(\mu, \tilde{\mu}) \,  \le \, 142 \,   d^{1/2} \, c^{-1} \, \left( L/K \right)^{1/2}   \,L^{1/4} \, h^{3/2} \;. 
\end{equation}
Remarkably, this upper bound only requires the assumption that $U$ is $K$-strongly convex and $L$-gradient Lipschitz. The proof of Theorem \ref{thm:asymptotic_bias} rests on the proof of $L^2$-accuracy of the sMC integration scheme given in Lemma~\ref{lem:smc_strong_accuracy}.

  %, which suggests that sMC might also be a useful tool for AVF.

To obtain the stated complexity guarantee, the hyperparameters are tuned as follows: 
\begin{equation}  \label{eq:tuning} \scriptsize
\begin{aligned}
&\text{duration $T \,\propto  \, L^{-1/2}$,}~~ \text{time step $ h \, \propto \, \left( \eps d^{-1/2} c \left(\frac{L}{K} \right)^{-1/2} L^{-1/4} \right)^{2/3} $,~~and~steps $ m  \, \propto \, c^{-1} \log\left( \frac{\W^2(\mu, \nu)}{\eps}\right)^+$}.
\end{aligned}
\end{equation}
For clarity, numerical prefactors are suppressed here, but are fully worked out in Theorem \ref{thm:uHMC_complexity} and Remark \ref{rmk:complexity}.
With this choice of hyperparameters, the $\W^2$-contraction rate in \eqref{intro:w2_contr} reduces to $c \propto K/L$, and we find that
\begin{align*} 
\W^2(\nu \tilde{\pi}^m, \mu) &\le  \W^2( \nu \tilde{\pi}^m, \tilde{\mu}) + \W^2( \tilde{\mu}, \mu) \\
& \overset{\eqref{intro:w2_contr}}{\le}\, e^{-cm} \W^2 (\nu, \tilde{\mu}) + \W^2( \tilde{\mu}, \mu) \, \le  \,e^{-cm} \W^2 (\nu, \mu) + 2\W^2( \tilde{\mu}, \mu)  \overset{\eqref{intro:w2_asymp_bias}}{\le} \eps \;.
\end{align*}
Since the total number of gradient evaluations is $m \times T/h$,  we obtain the complexity guarantee given in \eqref{smc_complexity}.

\subsection{Organization of the Paper}

The rest of the paper is organized as follows.  Section~\ref{sec:uHMC} contains a definition of the new uHMC algorithm with sMC time integration (Definition~\ref{def:uHMC_sMC}); the assumptions on $U$ (Assumption~\ref{A1234}); a theorem on $L^2$-Wasserstein Contractivity (Theorem~\ref{thm:contractivity}); a theorem on $L^2$-Wasserstein Asymptotic Bias (Theorem~\ref{thm:asymptotic_bias}); and a $L^2$-Wasserstein Complexity Guarantee (Theorem~\ref{thm:uHMC_complexity}).  Section~\ref{sec:proofs} contains detailed proofs.
An Appendix is included on: (i) `adjustable' randomized time integrators which can be directly incorporated into the proposal of Metropolis-adjusted HMC (see Appendix~\ref{sec:aHMC}); and (ii) duration randomization which  has a similar effect as partial momentum refreshment (see Appendix~\ref{sec:dr_uHMC}).

\smallskip

We conclude this introduction by remarking that randomized time integration might also be useful in conjunction with either time step or duration adaptivity to  deal with multiscale features in the target distribution --- as in \cite{kleppe2022, hoffman2022tuning,HoGe2014}.

\section*{Acknowledgements}
We wish to acknowledge Persi Diaconis for his support and encouragement of this collaboration. We also acknowledge
Bob Carpenter, 
Tore Kleppe, Pierre Monmarch\'{e},  Stefan Oberd\"{o}rster, and Lihan Wang for  feedback on a previous version of this paper. N.~Bou-Rabee has been supported by the National Science Foundation under Grant No.~DMS-2111224. M. Marsden also acknowledges his advisors Persi Diaconis and Lexing Ying for their support.

\section{uHMC with sMC time integration}

\label{sec:uHMC}

\subsection{Notation}

Let $\PM(\mathbb{R}^d)$ denote the set of all probability measures on $\mathbb{R}^d$, and denote by  $\PM^p(\mathbb{R}^d)$ the subset of probability measures on $\mathbb{R}^d$ with finite $p$-th moment.    Denote the set of all couplings of $\nu, \eta \in \PM(\mathbb{R}^d)$ by $\J(\nu,\eta)$. 
For $\nu,\eta \in \PM^p(\mathbb{R}^d)$,
define the $L^p$-Wasserstein distance by
\[
\W^p(\nu,\eta) \ := \ \left( \inf \Big\{ E\left[| X - Y |^p  \right] ~:~ \law(X, Y) \in \J(\nu,\eta) \Big\} \right)^{1/p} \;.
\] 

\subsection{Definition of the uHMC Algorithm with sMC time integration}

 Unadjusted Hamiltonian Monte Carlo (uHMC) is an MCMC method for approximate sampling from a `target' probability distribution on $\mathbb{R}^d$ of the form
\begin{equation}\label{eq:mu}
\mu (dx)= \mathcal Z^{-1}\,\exp (-U(x))\, dx \;, \quad \text{$\mathcal Z=
\int_{\mathbb{R}^d} \exp (-U(x))\,dx$}  \;,
\end{equation}
where $U: \mathbb{R}^{d} \to \mathbb{R}_{\ge 0}$ is assumed to be
a continuously differentiable function
such that $\mathcal Z<\infty$.  The function $U$ is termed  `potential energy' and $-\nabla U$ is termed `potential force' since it is a force derivable from a potential.

First we recall the standard uHMC algorithm with Verlet time integration and complete momentum refreshment. The standard algorithm generates a Markov chain on $\mathbb{R}^d$ using: (i) a  Verlet time integrator for the Hamiltonian flow corresponding to the unit mass Hamiltonian $H(x,v) = (1/2) |v|^2 + U(x)$; and (ii) an i.i.d.~sequence of random initial velocities $(\xi^k)_{k \in \mathbb{N}_0} \overset{i.i.d.}{\sim} \mathcal{N}(0,I_d)$.  We want to highlight that there are exactly two hyper-parameters that need to be specified in this algorithm: the duration $T>0$ of the Hamiltonian flow and the time step size $h \ge 0$; and for simplicity of notation, we often assume $T/h \in \mathbb{Z}$ when $h>0$, which implies that $h \le T$.  Let $\{ t_i := i h \}_{i \in \mathbb{N}_0}$ be an evenly spaced time grid.  This grid partitions time into subintervals $\{[t_i,t_{i+1})\}_{i \in \mathbb{N}_0}$ termed `strata'.  In the $(k+1)$-th uHMC transition step, a Verlet time integration is performed with initial position given by the $k$-th step of the chain and initial velocity given by $\xi^k$. The  $(k+1)$-th state of the chain is then the final position computed by Verlet. Recall that Verlet approximates the Hamiltonian flow using: (i) a piecewise quadratic approximation of positions which can be interpolated by a quadratic function of time on each stratum $[t_i, t_{i+1})$; and (ii) a deterministic trapezoidal quadrature rule of the time integral of the potential force over each stratum $[t_i, t_{i+1})$ to update the velocities.

However, since in uHMC we are almost exclusively interested in a stochastic notion of accuracy of the numerical time integration \cite[Theorem 3.6]{BouRabeeSchuh2023}, it is quite natural to instead use a randomized time integrator for the Hamiltonian flow.  This paper suggests one such randomized integration strategy.  %, which involves a very minor modification of the Verlet time integrator, and hence, is easy to implement.  
The basic idea is to replace the trapezoidal quadrature rule used by Verlet in each stratum with Monte Carlo.  This construction will substantially relax the regularity requirements on the target density.  The resulting integration scheme is an instance of a stratified Monte Carlo (sMC) time integrator. To precisely define this variant of uHMC, in addition to the i.i.d.~sequence of random initial velocities $(\xi^k)_{k \in \mathbb{N}_0}$, we define an independent sequence of sequences $(\mathcal{U}^k_i)_{i,k  \in \mathbb{N}_0}$ whose terms are independent uniform random variables $\mathcal{U}^k_i \sim \operatorname{Uniform}(t_i,t_{i+1})$.  A key ingredient in the algorithm is the \emph{sMC flow} $(\tilde{Q}_t^k(x,v), \tilde{P}_t^k(x,v))$ from $(x,v) \in \mathbb{R}^{2 d}$ defined by 
\begin{equation}
\label{sMC}
	\frac{d}{dt} \tilde{Q}_t^k   \, = \,   \tilde{P}_{t}^k , \qquad   \frac{d}{dt} \tilde{P}_t^k \, = \, - \nabla
	U(\tilde{Q}_{\lb{t}}^k + (\tau_t^k - \lb{t}) \tilde{P}_{\lb{t}}^k )
\end{equation}
where we introduced the floor (resp.~ceiling) function to the  nearest time grid point less (resp.~greater) than time $t$, i.e., \begin{equation}
\label{eq:round}
\lb{t}= \max\{ s\in h\Z \ : \ s\leq t \} \;, \quad \ub{t}= \min\{ s\in
h\Z \ : \ s\geq t \} \quad \text{ for $h>0$} \;,
\end{equation}
and the random temporal point in the $i$-th stratum
\[
\tau_t^k =  \mathcal{U}_i^k \qquad \text{for $t \in [t_i, t_{i+1})$} \;,
\] as illustrated below.   
 \begin{center}
 \begin{tikzpicture}[scale=1.5]
\draw[-, thick](0,0.0) -- (8,0.0);
\node[black,scale=1.] at (2,-0.35) {$t_{i-1}$};
\node[black,scale=1.] at (4,-0.35) {$t_{i}$};
\node[black,scale=1.] at (6,-0.35) {$t_{i+1}$};
\node[black,scale=1.] at (2.7,0.35) {$\mathcal{U}_{i-1}^k$};
\node[black,scale=1.] at (5.25,0.35) {$\mathcal{U}_i^k$};
\node[black, scale=1.5,fill=white] at (7.75,0.0) {$\dotsc$};
\node[black, scale=1.5,fill=white] at (0.25,0.0) {$\dotsc$};
\foreach \x in {2.,4.,6.} 
{
\filldraw[color=black,fill=black] (\x,0) circle (0.12); 
};
\filldraw[color=black,fill=white] (2.7,0) circle (0.12); 
\filldraw[color=black,fill=white] (5.25,0) circle (0.12); 
\end{tikzpicture}
\end{center}

\begin{minipage}{\textwidth}
\begin{minipage}{0.54\textwidth}
\noindent
By integrating \eqref{sMC}, note that $\tilde{Q}_t^k$ is a piecewise quadratic function of time that interpolates between the positions $\{  \tilde{Q}_{t_i}^k \}$ and satisfies
$\frac{d}{dt} \tilde{Q}_t^k  = \tilde{P}_{t}^k$ and $\left. \frac{d^2}{dt^2} \tilde{Q}_t^k  \right|_{t=t_i+} = - \nabla
	U(\tilde{Q}_{t_i}^k + (\mathcal{U}_i^k - t_i) \tilde{P}_{t_i}^k )$, while $\tilde{P}_t^k$ is a piecewise linear function of time that interpolates between the velocities  $\{  \tilde{P}_{t_i}^k  \}$.

%A cubic piecewise polynomial interpolation could be used such that $\left. \frac{d}{dt} q_t  \right|_{t=t_k+} = \left. \frac{d}{dt} q_t  \right|_{t=t_k-}$, but this would be more complicated to analyze.  
\end{minipage}
\begin{minipage}{0.4\textwidth}
\begin{center}
\begin{tabular}{c}
\includegraphics[width=\textwidth]{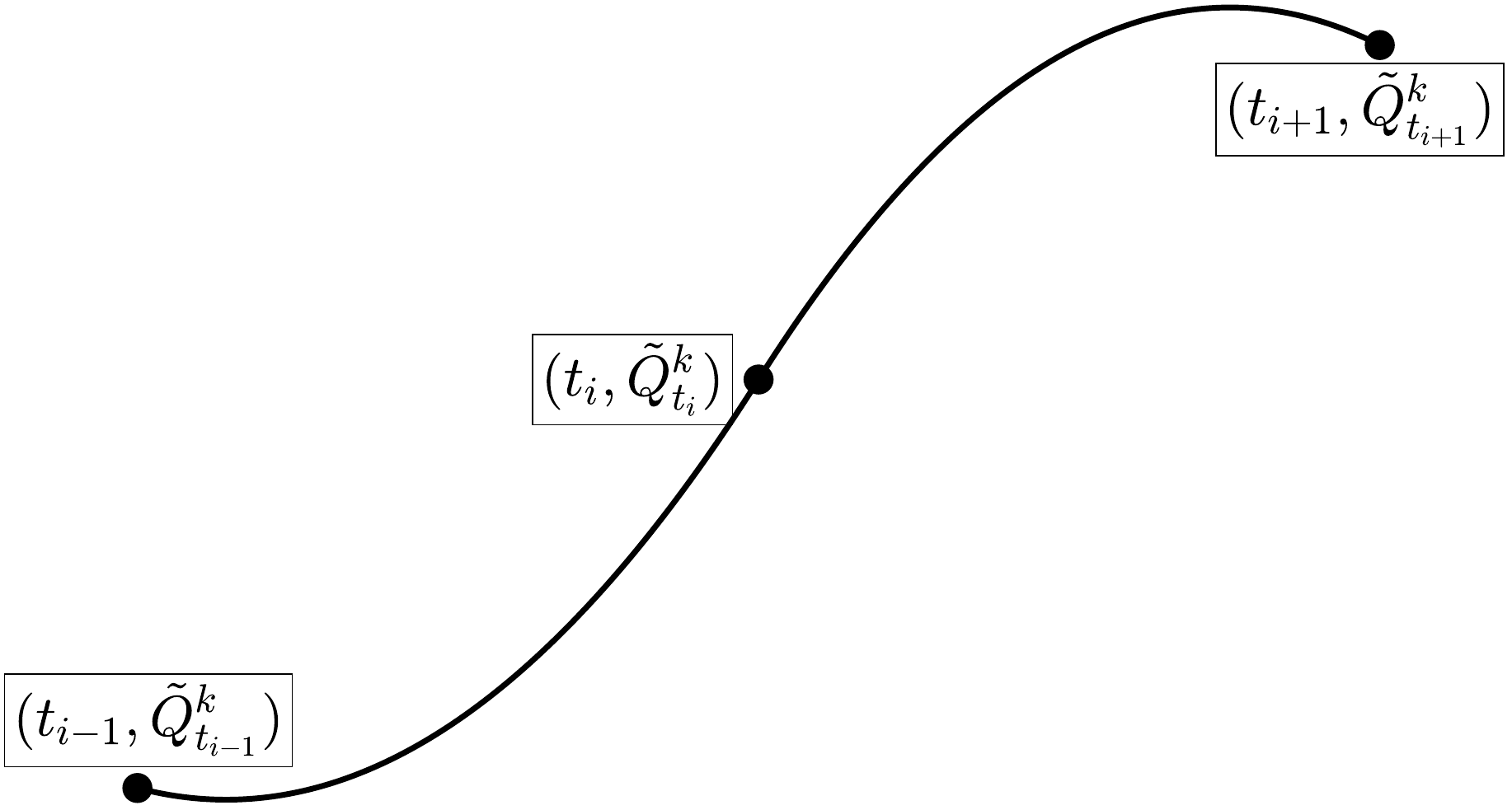} \\
{\em piecewise quadratic interpolation of positions}
\end{tabular}
\end{center}
\end{minipage}
\end{minipage}

\smallskip

For $h=0$, we set $\lb{t}=\ub{t}=t$, drop the tildes in the notation, and since the corresponding flow is deterministic, we use lower case letters  to denote the \emph{exact flow} $({q}_t(x,v), {v}_t(x,v))$ which satisfies
\begin{equation}
    \label{eq:Hamflow}
     \frac{d}{dt} {q}_t  \,  =\,     {p}_t ,\qquad  \frac{d}{dt} {p}_t \,  =\, - \nabla
	U({q}_t) \;.
\end{equation}
On the time grid, the sMC flow is an unbiased estimator of the \emph{semi-exact flow} $(\bar{q}_t(x,v),\bar{p}_t(x,v))$ which satisfies
\begin{equation}
\label{semi-exact}
	\frac{d}{dt} \bar{q}_t   \, = \,  \bar{p}_t , \qquad   \frac{d}{dt} \bar{p}_t \, = \, - \frac{1}{h} \int_{\lb{t}}^{\lb{t}+h}  \nabla
	U(\bar{q}_{\lb{t}} + (s-\lb{t}) \bar{p}_{\lb{t}} ) ds \;. 
\end{equation}
 The semi-exact flow plays a key role in \S\ref{sec:asymptotic_bias} to quantify the $L^2$-accuracy of the sMC flow, and is itself somewhat related to the Average Vector Field method \cite{quispel2008new}.

\medskip 

With this notation, the chain $(\tilde X^k)_{k \in \mathbb{N}_0}$ corresponding to uHMC with sMC time integration is defined as follows.

\begin{definition}[\textbf{uHMC with sMC time integration}]
Given an initial state $x \in \mathbb{R}^d$, a duration hyperparameter $T>0$, and time step size hyperparameter $h \ge 0$ with $T/h \in \mathbb{Z}$ for $h > 0$, define $\tilde{X}^0(x) := x$ and \[
\tilde{X}^{k+1}(x) := \tilde{Q}_T^{k}( \tilde{X}^{k}(x), \xi^{k} ) \quad \text{for $k \in \mathbb{N}_0$} \;.
\] 
Let $\tilde{\pi}(x,A) = P[ \tilde{X}^1(x) \in A ]$ denote the corresponding one-step transition kernel.  
\label{def:uHMC_sMC}
\end{definition}

For $h=0$, we recover \emph{exact HMC}. 
In this case, we drop all tildes in the notation, i.e., the $k$-th transition step is denoted by $X^k(x)$, and the corresponding transition kernel is denoted by $\pi$. The target measure $\mu$ is invariant under $\pi$, because the exact flow  preserves the Boltzmann-Gibbs probability measure on $\mathbb R^{2d}$ with density proportional to $\exp (-H(x,v))$, and $\mu$ is the first marginal of this measure. When $h>0$, and under certain conditions (detailed below), $\tilde{\pi}$ has a unique invariant probability measure denoted by $\tilde{\mu}$, which typically approaches $\mu$ as $h \searrow 0$.    
In the sequel, and for the sake of brevity, uHMC refers to  \emph{uHMC  with sMC time integration}.  

\smallskip

In Appendix~\ref{sec:aHMC}, we detail a novel adjustable HMC algorithm that employs randomized time integration to approximate the exact Hamiltonian flow.  Extending the sMC time integrator to be Metropolis-adjustable (see Definition~\ref{defn:adjustable}) is a challenging task.   Our approach involves constructing a randomized time integrator by randomly selecting from a parametric family of time integrators, each of which is reversible and volume-preserving.     Remarkably the accept/reject rule for the adjustable randomized time integrator, similar to that in standard HMC, depends solely on the change in energy incurred along a composition of these randomly selected reversible and volume-preserving integrators.  This result is particularly surprising given that the composition of reversible integrators is not necessarily reversible, whereas standard HMC requires a reversible proposal move (see, e.g., Section 5.3 of \cite{BoSaActaN2018}).   To clarify this surprising result, we provide a detailed but concise explanation in Appendix~\ref{sec:aHMC}.

\subsection{Assumptions}

%Let $\mathbf{H} \equiv \differential^2 U$ denote the Hessian of $U$ and $\mnorm{\cdot}$ denote the operator norm.
To prove our main results, we assume the following.

\begin{assumption} 
\label{A1234} The potential energy function $U: \mathbb{R}^d \to \mathbb{R}$ is continuously differentiable and satisfies:
\begin{enumerate}[label=\textbf{A.\arabic*}]
\item $U$ has a global minimum at $0$ and $U(0)=0$. \label{A1}
\item $U$ is $L$-gradient Lipschitz continuous, i.e., there exists $L >0$ such that \[
|\nabla U(x) - \nabla U(y)| \ \le \ L \, |x-y| \quad \text{for all $x,y \in \mathbb{R}^d$.} \] \label{A2}
\vspace{-0.25cm}
\item $U$ is $K$-strongly convex, i.e., there exists $K>0$ such that \[
( \nabla U(x) - \nabla U(y) ) \cdot (x - y) \ \ge \ K | x - y |^2 \quad \text{for all $x,y \in \mathbb{R}^d$.}
\] \label{A3}
\end{enumerate}
\end{assumption}

Assumptions \ref{A1}-\ref{A3} imply $\W^2$-contractivity of the transition kernel of uHMC; see Theorem~\ref{thm:contractivity} below.  By the Banach fixed point theorem, contractivity implies existence of a unique invariant probability measure of $\tilde{\pi}$ \cite[Theorem 2.9]{BouRabeeEberleLectureNotes2020}. The $\W^2$-asymptotic bias of this invariant measure is upper bounded in Theorem~\ref{thm:asymptotic_bias}. 

\begin{remark} \label{rmk:remark1}
Under \ref{A1}-\ref{A3}, using a  quadratic Foster-Lyapunov function argument\footnote{See Proposition 1(ii) of \cite{durmus2019high}.}, it can be shown that the target distribution satisfies \[
\int_{\mathbb{R}^d}  |x| \mu(dx) \ \le \ \left( \int_{\mathbb{R}^d}  |x|^2 \mu(dx) \right)^{1/2}  \ \le \  \left( \frac{d}{K} \right)^{1/2} 
\] where in the first step we used Jensen's inequality. The bound is sharp in the sense that it is attained by a centered Gaussian random variable $\xi$ with $E | \xi |^2 = d/K$.
\end{remark}

\begin{remark} \label{rmk:cocoercivity}
If $U$ is continuously differentiable, convex, and $L$-gradient Lipschitz, then $\nabla U$ satisfies the following `co-coercivity' property \begin{equation} \label{eq:cocoercive}
|\nabla U(x) - \nabla U(y)|^2 \ \le \ L \, (\nabla U(x) - \nabla U(y)) \cdot (x- y) \;, \quad \text{for all $x, y \in \mathbb{R}^d$} \;. 
\end{equation}
This  property  plays a crucial role in proving a sharp $\W^2$-contraction coefficient for the uHMC transition kernel in the globally strongly convex setting.
%For a proof of co-coercivity of the gradient, see, e.g., Theorem 5.8 of \cite{beck2017first}.
\end{remark}

\subsection{$L^2$-Wasserstein Contractivity}

Let $(\tilde{Q}_t(x,v), \tilde{P}_t(x,v))$ be a single realization of the sMC flow satisfying \eqref{sMC} from the initial condition $(x,v) \in \mathbb{R}^{2d}$ with a random sequence of independent temporal sample points $(\mathcal{U}_i )_{i \in \mathbb{N}_0}$ such that $\mathcal{U}_i \sim \operatorname{Uniform}(t_i,t_{i+1})$.
When $U$ is $K$-strongly convex and $L$-gradient Lipschitz, the exact flow from different initial positions but synchronous initial velocities is itself contractive if $LT^2 \le 1/4$ \cite[Lemma 6]{chen2019optimal}.\footnote{ Under $LT^2 \le \min(1/4, K/L)$, Mangoubi and Smith first obtained a similar result \cite{mangoubi2017rapid}.}  Analogously, if $LT^2 \le 1/8$ and the time step size additionally satisfies $h \le T$, the following lemma states that $|\tilde{Q}_T(x,v)-\tilde{Q}_T(y,v)|^2$ is almost surely contractive.  

\begin{lemma}[Almost Sure Contractivity of sMC Time Integrator] \label{lem:contr_SMC}
Suppose that \ref{A1}-\ref{A3} hold.  Let $T>0$   satisfy: \begin{align}
 L T^2 & \, \le \,  1/8 \;, \label{eq:CT} \end{align}
and $h \ge 0$ satisfy $T/h \in \mathbb{Z}$ if $h>0$.  Then for all $x, y, v \in \mathbb{R}^d$, \begin{equation}
|\tilde{Q}_T(x,v) - \tilde{Q}_T(y,v)|^2 \ \le \ \left( 1 - K \, T^2 \, / \, 3 \right) \ |x - y |^2 \quad \text{almost surely} \;.
\end{equation}
\end{lemma}

The proof of Lemma~\ref{lem:contr_SMC} is deferred to Section~\ref{sec:w2contractivity}.  By synchronously coupling both the random initial velocities and the random temporal sample points in two copies of uHMC starting at different initial conditions, and applying Lemma~\ref{lem:contr_SMC}, we obtain the following.  

\begin{theorem}[$\W^2$-Contractivity of uHMC] \label{thm:contractivity}
Suppose that \ref{A1}-\ref{A3} hold. Let $T>0$ and $h \ge 0$ be such that \eqref{eq:CT} holds with $T/h \in \mathbb{Z}$ if $h>0$.  Then for any pair of probability measures $\nu, \eta \in \PM^2(\mathbb{R}^d)$, \begin{align}
\W^2(\nu \tilde{\pi}^m, \eta \tilde{\pi}^m) \, &\le \, (1 - c)^m \, \W^2(\nu, \eta) \qquad \text{where} \\
c \, & = \, K \, T^2 \, / \, 6 \;.
\end{align}
\end{theorem}

\begin{proof}
Let $\omega$ be an arbitrary coupling of $\nu, \eta \in \PM^2(\mathbb{R}^d)$.  By the coupling characterization of the $\W^2$-distance,  \begin{align*}
\W^2(\nu \tilde{\pi}, \eta \tilde{\pi})^2 \, &\le \,  E_{(X,Y) \sim \omega, \xi \sim \mathcal{N}(0,I_d)} |\tilde{Q}_T(X,\xi) - \tilde{Q}_T(Y,\xi)|^2  \\
\, &\overset{\text{Lem.~\ref{lem:contr_SMC}}}{\le} \, \left( 1 - K \, T^2 \, / \, 3 \right) E_{(X,Y) \sim \omega}  |X - Y |^2  \;.
\end{align*}
Taking the infimum over all couplings $\omega$ and using the inequality $\sqrt{1- \mathsf{a}} \le 1 - \mathsf{a}/2$  for  $\mathsf{a} \in [0,1)$, we obtain the required result for $m=1$. By iterating this inequality,  the result extends to all $m \in \mathbb{N}$.
\end{proof}

Note that the $\W^2$-contraction coefficient in Theorem~\ref{thm:contractivity} is uniform in the time step size, and as $h \searrow 0$, recovers (up to a numerical prefactor) the sharp $\W^2$-contraction coefficient of exact HMC \cite[Lemma 6]{chen2019optimal}.

\subsection{$L^2$-Wasserstein Asymptotic Bias}
\label{sec:asymptotic_bias}

As emphasized in previous works \cite{BouRabeeSchuh2023,DurmusEberle2024}, an apt notion of accuracy of the underlying time integrator in unadjusted HMC (and other inexact MCMC methods) is a stochastic one, e.g., $L^2$-accuracy.  Remarkably, the sMC time integrator is $3/2$-order $L^2$-accurate without higher regularity assumptions such as Lipschitz continuity of the Hessian of $U$.

\begin{lemma}[$L^2$-accuracy of sMC Time Integrator]
\label{lem:smc_strong_accuracy}
Suppose that \ref{A1}-\ref{A2} hold. Let $T>0$ satisfy $L T^2 \le 1/8$,  and let $h>0$ satisfy $T/h \in \mathbb{Z}$.  Then for any $x \in \mathbb{R}^d$ and $k \in \mathbb{N}_0$ such that $t_k \le T$, \begin{equation}
   \left( E\left[ |\tilde{Q}_{t_k}(x,v) - q_{t_k}(x,v)|^2 \right] \right)^{1/2} \, \le \,    71  \, (|v| + \sqrt{L} |x| ) \, L^{1/4} h^{3/2} \;.  \label{str_err_smc}
\end{equation}
\end{lemma}

Note that \ref{A3} is not assumed in Lemma~\ref{str_err_smc}. The $3/2$-order of $L^2$-accuracy of the sMC time integrator is numerically verified in Figure~\ref{fig:strong_accuracy}. 

\begin{proof}[Proof of Lemma~\ref{lem:smc_strong_accuracy}] The proof of  $L^2$-accuracy of the sMC integrator is carried out in two steps. In Lemma~\ref{lem:smc_strong_accuracy_semi_exact}, we compare the sMC flow to the semi-exact flow. In Lemma~\ref{lem:semi_exact_accuracy}, we compare the semi-exact flow to the exact flow.  Using the triangle inequality and the bound $L^{1/4}h^{1/2} \leq 1/\sqrt{2}$, we obtain the required result with the numerical prefactor derived from summing the prefactors in \eqref{str_err_smc_semi_exact} and \eqref{semi_exact_accuracy} as follows: \[ \sqrt{2} e^{31/8} + 2 e^{5/8} L^{1/4} h^{1/2} \le\sqrt{2} (e^{31/8} + e^{5/8}) \le 71 \;. \] \end{proof}

\begin{remark}
For comparison, under only the assumption that $U$ is $L$-gradient Lipschitz, the order of accuracy of Verlet integration often drops from second to first order \cite[Theorem 8]{BouRabeeSchuh2023}.  This drop in accuracy is  due to the use of a trapezoidal approximation of the integral of the force $-\nabla U$, and it is well known that the trapezoidal rule typically drops an order of accuracy if the integrand does not have a bounded second derivative.  In turn, the accuracy of the integration scheme affects the asymptotic bias between the invariant measure of uHMC and the target distribution.  Thus, a smaller time step size is needed to resolve the asymptotic bias of uHMC, which increases the complexity of the algorithm.
\end{remark}

\begin{remark} \label{rmk:msefundamental}
 The $3/2$-order of $L^2$-accuracy of the sMC time integrator is reminiscent of the classical Fundamental Theorem for $L^2$-Convergence of Strong Numerical Methods for SDEs, which roughly states: if $p_1 \ge p_2 + 1/2 $ and $p_2 > 1/2$ are the order of mean and mean-square accuracy (respectively), then the $L^2$-accuracy of the method is of order $p_2 - 1/2$ \cite[Theorem 1.1.1]{Milstein2021}.  For strong numerical methods for SDEs, this $(p_2 - 1/2)$-order (as opposed to $(p_2 - 1)$-order) is due to cancellations in the $L^2$ error expansion, due to independence of the Brownian increments.   Consequently, the expectation of cross terms that involve the Brownian increments can vanish because they have zero mean. Here the cancellations (to leading order) occur because of independence of the sequence of random temporal sample points used by the sMC time integrator.  Specifically, what happens is that the expectation of the random potential force appearing in the cross terms turns into an average of the force over the  stratum, which confers higher accuracy in these cross-terms.  
  Turning this heuristic  into a rigorous proof relies on comparison to a semi-exact flow, which uses the mean force to update the position and velocity; see Lemmas~\ref{lem:smc_strong_accuracy_semi_exact} and~\ref{lem:semi_exact_accuracy} for details.
\end{remark}

\begin{figure}[ht!]
\centering
\includegraphics[width=0.49\textwidth]{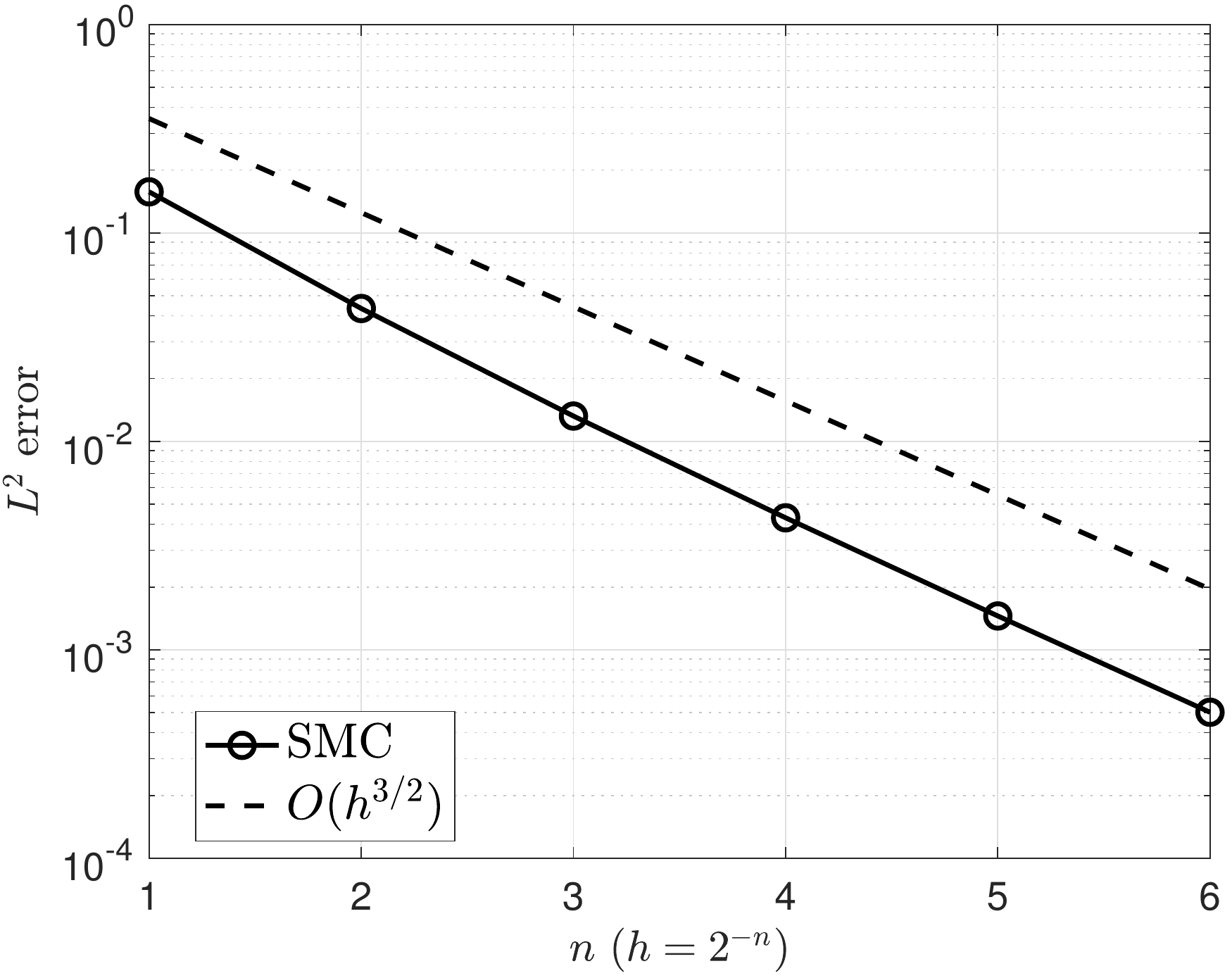}  \hfill
\includegraphics[width=0.49\textwidth]{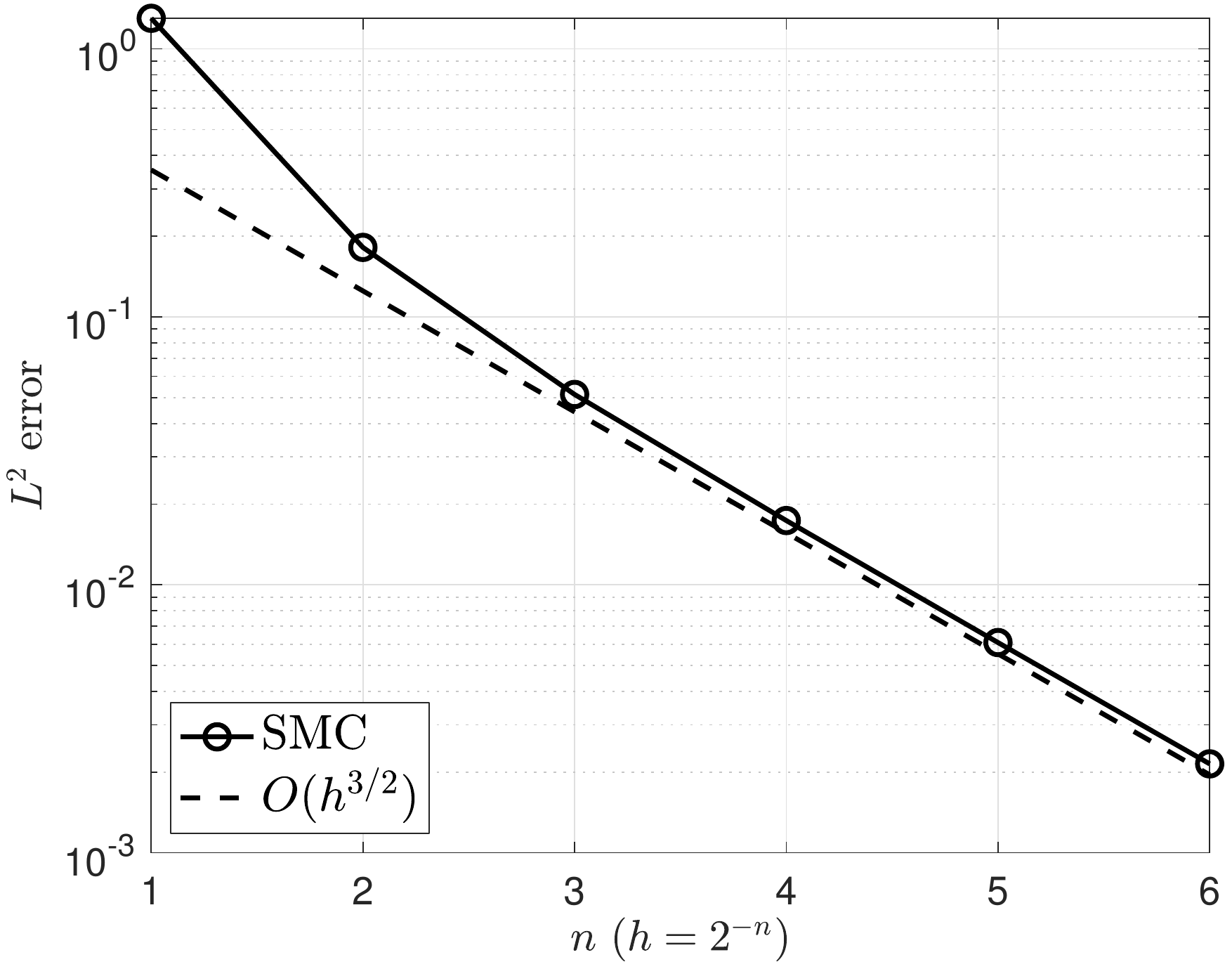}
\caption{ \small {\bf $L^2$-Accuracy Verification.} Left Image: A plot of the $L^2$-error in $(x,v)$-space of the sMC time integrator for the linear oscillator with Hamiltonian $H(x,v) = (1/2) (v^2 + x^2)$.  Right Image: A plot of the $L^2$-error in $(x,v)$-space of the sMC time integrator for a double-well system with Hamiltonian $H(x,v) = (1/2) (v^2 + (1-x^2)^2)$. Both simulations have initial condition $(2,1)$ and unit duration. The time step sizes tested are $2^{-n}$ where $n$ is given on the horizontal axis. The dashed curve is $2^{-3 n/2} = h^{3/2}$ versus $n$. }
\label{fig:strong_accuracy}
\end{figure}

Additionally, assuming \ref{A3} holds, and using a triangle inequality trick \cite[Remark 6.3]{mattingly2010convergence}, the $L^2$-accuracy bound in Lemma~\ref{lem:smc_strong_accuracy} combined with the $\W^2$-Contractivity in Theorem~\ref{thm:contractivity} can be used to bound the $\W^2$-asymptotic bias.  

\begin{theorem} [$\W^2$-Asymptotic Bias of uHMC] \label{thm:asymptotic_bias}
Suppose that \ref{A1}-\ref{A3} hold.  Let $T>0$ and $h \ge 0$ be such that \eqref{eq:CT} holds with $T/h \in \mathbb{Z}$ if $h>0$. Additionally, assume $LT^2 \le 1/8$. Then \begin{equation*}
    \W^2(\mu, \tilde{\mu}) \, \le \, 142 \, \,  d^{1/2} \, c^{-1} \, \left( L/K \right)^{1/2}   \,L^{1/4} h^{3/2}  \;.
\end{equation*}
\end{theorem}

\begin{proof}
By the triangle inequality,  \begin{align}
&\W^2(\mu, \tilde{\mu}) = \W^2(\mu \pi, \tilde{\mu} \tilde{\pi}) \le \W^2(\mu \pi, \mu \tilde{\pi}) + \W^2(\mu \tilde{\pi}, \tilde{\mu} \tilde{\pi}) \nonumber \\
&\quad \overset{Thm.~\ref{thm:contractivity}}{\le} \W^2(\mu \pi, \mu \tilde{\pi}) + (1-c) \W^2(\mu, \tilde{\mu})  \implies \W^2(\mu, \tilde{\mu}) \le c^{-1} \W^2(\mu \pi, \mu \tilde{\pi})  \;. \label{eq::w2_asymp_bias}
\end{align} 
 Employing now Lemma \ref{lem:smc_strong_accuracy}, Remark \ref{rmk:remark1}, and $L/K \geq 1$ gives \begin{align*}
\W^2(\mu \pi, \mu \tilde{\pi})^2 &\overset{\text{Lem.~\ref{lem:smc_strong_accuracy}}}{\le} E_{X \sim \mu, \xi \sim \mathcal{N}(0, I_d)} \left[ |\tilde{Q}_{t_k}(X,\xi) - q_{t_k}(X,\xi)|^2 \right]  \\ &\overset{\text{Rmk.~\ref{rmk:remark1}}}{\le} 2 \cdot 71^2 d ( 1 + \frac{L}{K}) L^{1/2} h^3 \, \le \,  4 \cdot 71^2  d \frac{L}{K} L^{1/2} h^3 \;.
 \end{align*}
Inserting this result back into \eqref{eq::w2_asymp_bias} gives the required  bound.
\end{proof}

%\begin{remark} Note the improvement in the $\W^2$-asymptotic bias over uHMC with Verlet time integration. This improvement can be understood via the classical Fundamental Theorem of $L^2$-Convergence of Strong Numerical Methods for SDEs \cite[Theorem 1.1.1]{Milstein2021}.  This theorem highlights that the expansion of the squared $L^2$-error of a stochastic numerical method can lead to cancellations of cross terms.  For strong numerical methods for SDEs, this cancellation of cross terms occurs because of the independence of the Brownian increments used at each integration step. Consequently, the expectation of cross terms that involve the Brownian increments can vanish because they have zero mean.  For the sMC time integrator, a similar cancellation  can occur, which is analogously due to independence of the sequence of random temporal sample points.  Specifically, what happens is that the expectation of the random potential force $\Tilde{F}_{t_i}$ appearing in the cross terms turns into an average of the potential force over the $i$-th stratum, which confers higher accuracy in these cross-terms.  Turning this heuristic  into a rigorous proof relies on comparison to a `semi-exact' flow, which uses the mean of $\Tilde{F}_{t_i}$ to update the position and velocity; see Lemmas~\ref{lem:smc_strong_accuracy_semi_exact} and~\ref{lem:semi_exact_accuracy} for details.\end{remark}

\subsection{$L^2$-Wasserstein Complexity}

In sum, Theorem~\ref{thm:contractivity} implies, for any $\nu \in \PM^2(\mathbb{R}^d)$, $\W^2(\tilde{\mu}, \nu \tilde{\pi}^m) \le e^{-c m} \W^2(\tilde{\mu}, \nu)$;  and Theorem~\ref{thm:asymptotic_bias} implies $\W^2(\mu,\tilde{\mu}) \le O(h^{3/2})$.  Together these two results imply $\W^2(\mu, \nu \tilde{\pi}^m)$ can be made arbitrarily small by choosing $h$ sufficiently small and $m$ sufficiently large. This complexity argument is quite standard for inexact MCMC methods \cite{DurmusEberle2024,BouRabeeSchuh2023}, but we briefly summarize it here for completeness.  

\begin{theorem}[$\W^2$-Complexity of uHMC] \label{thm:uHMC_complexity}
Suppose that \ref{A1}-\ref{A3} hold.  Let $T>0$ and $h \ge 0$ be such that \eqref{eq:CT}  holds with $T/h \in \mathbb{Z}$ if $h>0$.  Let $\nu \in \PM^2(\mathbb{R}^d)$.  Define $\Delta(m):=\W^2(\mu, \nu \tilde{\pi}^m)$, i.e., the $L^2$-Wasserstein distance to the target measure $\mu$ after $m$ steps of uHMC initalized at $\nu$.  For any $\varepsilon>0$, suppose that $h\ge 0$ and $m \in \mathbb{N}$ satisfy \[
m \, \ge \, m^{\star} := c^{-1} \log( 2 \W^2(\mu, \nu) / \varepsilon)^+ \;,  \quad h \,\le \, h^{\star} := e^{-4} \left(\frac{c \, (\varepsilon/2)}{d^{1/2} (L/K)^{1/2} L^{1/4} } \right)^{2/3} \;. 
\]  Then $\Delta(m) \le \varepsilon$.  
\end{theorem}

\begin{proof}
Let $m \ge m^{\star}$ and $h \le h^{\star}$.
By the triangle inequality, \begin{align*}
&    \Delta(m) \, \le \, \W^2(\mu, \tilde{\mu}) + \W^2( \tilde{\mu}, \nu \tilde{\pi}^m )  \overset{Thm.~\ref{thm:contractivity}}{\le} \, \W^2(\mu, \tilde{\mu}) + e^{- c \, m }\W^2( \tilde{\mu}, \nu ) \\
& \quad \le  2 \W^2(\mu, \tilde{\mu}) + e^{- c \, m }\W^2( \mu, \nu ) \overset{Thm.~\ref{thm:asymptotic_bias}}{\le} e^6  d^{1/2} c^{-1}    \left(\frac{L}{K} \right)^{1/2}   L^{1/4} h^{3/2}  + e^{- c \, m }\W^2( \mu, \nu ) \;.
\end{align*}
Since $m \ge m^{\star}$ and $h \le h^{\star}$, the required result follows.
\end{proof}

\begin{remark}[$\W^2$-Complexity Guarantee]  \label{rmk:complexity}
For any accuracy $\varepsilon>0$,
Theorem~\ref{thm:uHMC_complexity} states that  $m \ge m^{\star}$ uHMC transition steps with $h \le h^{\star}$ guarantee that $\Delta(m) \le \varepsilon$.  
To turn this into a complexity guarantee, we specify the duration, and in turn, estimate the corresponding number of gradient evaluations.  In particular, if one chooses the duration $T$ to saturate the condition $L T^{2} \le 1/8$ in \eqref{eq:CT}, i.e., $T=(8L)^{-1/2}$, then the $\W^2$-contraction coefficient reduces to $c=48^{-1} K/L$.  Since each uHMC transition step involves $T/h$ gradient evaluations, the corresponding number of gradient evaluations is  \begin{align*}
\scriptsize
\left(m~ \text{transition steps} \right) \times \left( T/h~\text{integration steps} \right)    =   
8^{7/6} \cdot 6^{5/3} \cdot e^4 \,   \left( \frac{d}{K} \right)^{1/3} \left( \frac{L}{K} \right)^{5/3} \left( \frac{\varepsilon}{2} \right)^{-2/3} \log\left( \frac{2 \W^2(\mu, \nu)}{ \varepsilon} \right)^+ \;,
\end{align*}
where explicitly $m = 48 L/K \log\left( \frac{2 \W^2(\mu, \nu)}{ \varepsilon} \right)^+$, $T = (8L)^{-1/2}$ and $h = K (\varepsilon/2)^{2/3} / ( 4 \cdot 6^{2/3} d^{1/3} e^4 L^{7/6} ) $.
\end{remark}

\subsection{Duration Randomization}

We have discussed the benefits of time integration randomization for uHMC with complete momentum refreshment.  Motivated by Monmarch\'{e}'s findings \cite{monmarche2022united}, we explore corresponding results for uHMC with partial momentum refreshment.  Since partial momentum refreshment eventually leads to a complete momentum refreshment after a random number of steps, we consider uHMC with duration randomization, as introduced in \S6 of \cite{BoSa2017}.

The duration-randomized uHMC process is an implementable pure jump process on phase space $\mathbb{R}^{2d}$, alternating between two types of jumps: (i) a single step of a randomized time integrator for the Hamiltonian flow; and (ii) a complete momentum refreshment.  This algorithm uses two hyperparameters: the time step $h$ and the mean duration $\lambda^{-1}$.  The existence of an explicit generator for the process facilitates detailed analyses of contractivity and asymptotic bias.

In Appendix~\ref{sec:dr_uHMC}, we analyze the complexity of the duration-randomized uHMC process in detail and highlight its improvements over the duration-fixed uHMC chain (see Remark~\ref{rmk:durationrandomization}).  The improvement in complexity due to duration randomization is analogous to that from time integration randomization, leading to a better $\W^2$-asymptotic bias.  For time integration randomization, this improvement arises from averaging over the random midpoint used per integration step (see Remark~\ref{rmk:msefundamental}).  For duration randomization, it stems from averaging over the random duration.

Theorem~\ref{thm:contr_mjp} states that the $\W^2$-contraction rate of this randomized uHMC process on phase space is $\gamma = K \lambda^{-1} / 10$.  The proof relies on a synchronous coupling of two copies of the duration-randomized uHMC process.  When well-tuned, the contraction rate becomes $\gamma \propto K^{1/2} (K/L)^{1/2}$. Although this rate appears better than the duration-fixed uHMC rate of $c \propto K/L$, the rates are essentially identical when viewed on the same infinitesimal time scale, i.e., $c/T \propto \gamma$.

To quantify the asymptotic bias, we couple the duration-randomized uHMC process with an `exact' counterpart.  This coupling, which is itself a jump process with generator $\mathcal{A}^C$ defined in \eqref{eq:AC_asympbias}, admits a Foster-Lyapunov function $\rho^2$  given by a quadratic `distorted' metric on phase space \cite{achleitner2015large}, defined in \eqref{eq:dist_metric}.   Using It\^{o}'s formula for jump processes the corresponding finite-time drift condition is used in Theorem~\ref{thm:asympbias_mjp} to quantify the $\W^2$-asymptotic bias of the duration-randomized uHMC process. By optimally tuning the hyperparameters, an improvement in the complexity of the duration-randomized uHMC process is found (see Remark~\ref{rmk:durationrandomization}).

\section{Proofs}

\label{sec:proofs}

In this section we provide the remaining ingredients needed to prove Theorem~\ref{thm:contractivity} and Theorem~\ref{thm:asymptotic_bias}.

\subsection{Proof of $L^2$-Wasserstein contractivity}

\label{sec:w2contractivity}

The following proof carefully adapts ideas from \cite{BouRabeeSchuh2023} and Lemma 6 of \cite{chen2019optimal}.  The main idea in the proof is to carefully balance two competing effects at the random temporal sample points where the potential force is evaluated: (i) strong convexity of $U$ (\ref{A3}); and (ii) co-coercivity of $\nabla U$ (Remark~\ref{rmk:cocoercivity}).

\begin{proof}[Proof of Lemma~\ref{lem:contr_SMC}]
Let $t \in [0,T]$. Introduce the shorthands \begin{equation}
 \label{eq:rand_xt_yt}
 \begin{aligned}
 x_t := \tilde{Q}_{\lb{t}}(x,v) + (\tau_t - \lb{t}) \tilde{P}_{\lb{t}}(x,v)   ~~\text{\&} ~~
 y_t := \tilde{Q}_{\lb{t}}(y,v) + (\tau_t - \lb{t}) \tilde{P}_{\lb{t}}(y,v) \;,
 \end{aligned}
 \end{equation}
 where recall that $\tau_t =  \mathcal{U}_i$ for $t \in [t_i, t_{i+1})$ and $(\mathcal{U}_i)_{i \in \mathbb{N}_0}$ is a sequence of independent random variables such that $ \mathcal{U}_i \sim \operatorname{Uniform}(t_i, t_{i+1})$.  Let $ z_t: = x_t - y_t $, \[
Z_t \, := \, \tilde{Q}_t(x,v) - \tilde{Q}_t(y,v) \;, ~~\text{and} ~~W_t \, := \, \tilde{P}_t(x,v) - \tilde{P}_t(y,v) \;.
 \] Let $A_t \, :=\, |Z_t|^2$, $B_t := 2 Z_t \cdot W_t$ and $a_t \, :=\, |z_t|^2$. Our goal is to obtain an upper bound for $A_t$.  To this end, define \[
\rho_t \, := \,  \phi_t \cdot z_t  \;, \qquad \phi_t \, := \,  \nabla U(x_t) - \nabla U(y_t)  \;,  
\] 
and note that by \ref{A2}, \ref{A3} and \eqref{eq:cocoercive}, \begin{equation} \label{ieq:rho}
    K a_t \, \overset{\ref{A3}}{\le} \, \rho_t  \,  \overset{\ref{A2}}{\le} \, L a_t \, \;, \qquad  |\phi_t|^2 \, \overset{\eqref{eq:cocoercive}}{\le} \, L \rho_t  \;, \qquad \text{for all $t \ge 0$} \;.
\end{equation} 
%In \eqref{ieq:rho}, we emphasize that the inequalities take place at the random positions $(x_t, y_t)$ in \eqref{eq:rand_xt_yt} where the potential force is evaluated.   
By \eqref{sMC}, note that \begin{align} 
		\frac{d}{dt} Z_t \, &= W_t \;,
		\qquad \frac{d}{dt} W_t \, = \, -  \phi_{t} \;. \label{eq:dotz_dotw}
\end{align}
 As a consequence of \eqref{eq:dotz_dotw}, a short computation shows that $A_t$ and $B_t$ satisfy
	\begin{align} \label{ivp_a}
		\frac{d}{dt} A_t & \ = \ B_t \;, \qquad \frac{d}{dt} B_t  \ = \   -  K A_t  + 2 |W_t|^2 + \epsilon_t \;,
		\end{align}
where $\epsilon_t \, := \,   K A_t  - 2 Z_t \cdot  \phi_t $. Introduce the shorthand notation \[
 s_{t-r} \, := \, \frac{\sin( \sqrt{ K} (t-r) )}{ \sqrt{ K}} \;, \quad  \text{and} \quad  c_{t-r} \, := \,  \cos(\sqrt{ K} (t-r) ) \;, 
 \] which satisfy $c_{t-r} = - \frac{d}{dr} s_{t-r}$.  By variation of parameters, 
\begin{align} 
   A_T & \, = \, c_T A_0 + \int_0^T s_{T-r} (2 |W_r|^2 + \epsilon_r) dr \;. \label{vop_at} 
    \end{align}
To upper bound the integral involving $|W_r|^2$ in \eqref{vop_at},   use \eqref{eq:dotz_dotw} and note that $W_0=0$, since the initial velocities in the two copies are synchronized.  Therefore, \begin{align} \label{eq:Wt2}
 |W_t|^2 \, = \,   \left|  \int_0^{t}   \phi_{s}  ds \right|^2 \ \, \le \,  t  \int_0^{t} | \phi_{s} |^2 ds
		\, \le \,  L t \int_0^{t}  \rho_{s} ds 	
\end{align}
where in the second step we used Cauchy-Schwarz, and in the last step, we used \eqref{ieq:rho}. Note that since $L T^2 \le 1/8$, $s_{t-r}$ is monotonically decreasing with $r$,  \begin{equation} \label{eq:stmr}
0 \le s_{t-r} \le s_{t-s} \;, \quad \text{for $s \le r \le t \le T \le \frac{\pi/2}{\sqrt{ K}} $} \;. \end{equation} Therefore, combining \eqref{eq:Wt2} and \eqref{eq:stmr}, and by Fubini's Theorem,   \begin{align}
\int_0^t s_{t-r} |W_r|^2 dr & \, \overset{\eqref{eq:Wt2}}{\le} \,
 L \int_0^t \int_0^r  r s_{t-r}   \rho_{s} ds dr
 \overset{\eqref{eq:stmr}}{\le} \,
 L \int_0^t \int_0^r  r s_{t-s}   \rho_{s} ds dr
\nonumber \\
& \ \ \le \,  L \int_0^t \int_s^t  r s_{t-s}  \rho_{s}  dr ds  \, \le \,  (L t^2 / 2) \int_0^t   s_{t-s} \rho_{s}  ds \;. \nonumber
\end{align} In fact, as a byproduct of this calculation, observe that \begin{equation}
(\int_0^t s_{t-r} |W_{\lb{r}}|^2 dr ) \vee ( \int_0^t s_{t-r} |W_r|^2 dr) \, \le \,  (L t^2 / 2) \int_0^t   s_{t-s} \rho_{s}  ds \;.  \label{ieq:wt2}
\end{equation}
To upper bound the integral involving $\epsilon_r$  in \eqref{vop_at}, note by \eqref{eq:dotz_dotw}
\[
Z_t \, = \,  Z_{\lb{t}} + (t - \lb{t}) W_{\lb{t}} - \frac{(t - \lb{t})^2}{2} \phi_t \;,
\]
and hence, $Z_t - z_t = (t  - \tau_t) W_{\lb{t}} - \frac{(t - \lb{t})^2}{2} \phi_t $ and \begin{align}    & \epsilon_t \, = \,  K A_t - 2 \rho_t  - 2 (Z_t - z_t) \cdot \phi_t  \, = \,  K A_t - 2 \rho_t  - 2 ( t - \tau_t) W_{\lb{t}} \cdot \phi_t + ( t- \lb{t})^2 |\phi_t|^2  \nonumber \\
& ~~~\, = \,   K a_{t}  - 2 \rho_{t}  - 2 ( t- \tau_t) W_{\lb{t}} \cdot \phi_t + (t-\lb{t})^2 | \phi_t|^2 \nonumber  \\
    &  \qquad +  K \left|(t - \tau_t) W_{\lb{t}} -  \frac{(t-\lb{t})^2}{2} \phi_t \right|^2  + 2  K \left( (t - \tau_t) W_{\lb{t}} - \frac{(t-\lb{t})^2}{2} \phi_t \right) \cdot z_{t}  \nonumber \\
&  \overset{\eqref{ieq:rho}}{\le}  K a_{t}  - 2 \rho_{t}  - 2 ( t- \tau_t) W_{\lb{t}} \cdot \phi_t + (t-\lb{t})^2 | \phi_t|^2   \nonumber \\
& \qquad +  K \left|(t - \tau_t) W_{\lb{t}} -  \frac{(t-\lb{t})^2}{2} \phi_t \right|^2 + 2  K (t - \tau_t) W_{\lb{t}} \cdot z_{t}     \nonumber \\
& ~  \le  - 2 \rho_{t} +   K ( 1+  K h^2 ) a_{t} + (2 h^2 +  K h^4 /2 )   | \phi_t|^2  + 2 (1 +   K h^2) |W_{\lb{t}}|^2   \label{ieq:epsi}
\end{align}
where in the last step we used Cauchy-Schwarz and Young's product inequality; and in the next to last step we used that $z_t \cdot \phi_t = \rho_t \ge 0$, which follows from \eqref{ieq:rho}.
 Inserting these bounds into the second term of \eqref{vop_at} yields\begin{align}
&  \int_0^T s_{T-r} (2 |W_r|^2 + \epsilon_r) dr \nonumber \\
&  \quad \overset{\eqref{ieq:epsi}}{\le}   \int_0^T s_{T-r} \left(2 |W_r|^2 + 2 (1+ K h^2) |W_{\lb{r}}|^2 -2 \rho_r + (2 h^2 +  K h^4 /2) |\phi_r|^2  +  K (1+ K h^2) a_r \right) dr 
\nonumber \\
& \quad \overset{\eqref{ieq:wt2}}{\le}    \int_0^T s_{T-r} \left(\left[  L T^2 (2 +  K h^2)  -2 \right]  \rho_{r}+ (2 h^2 +  K h^4 /2) |\phi_r|^2 +  K (1+ K h^2) a_r \right) dr 
\nonumber \\
&  \quad \overset{\eqref{ieq:rho}}{\le}   \int_0^{T} s_{T-r} \left( \left[ L h^2 (2 +  K h^2/2) + L T^2 (2 +  K h^2)  -2 \right]  \rho_{r}  +  K (1+ K h^2) a_{r}  \right) dr \;, \nonumber \\
& \quad  \overset{\eqref{ieq:rho}}{\le}   K \int_0^{T} s_{T-r} \left[ L h^2 (2 +  K h^2/2) + L T^2 (2 +  K h^2)  -2   +  (1+ K h^2) \right] a_{r}   dr \;. \nonumber
\end{align}
where in the last step we used that $K \le L$, $h \le T$, 
and $L T^2 \le 1/8$.  In fact, this upper bound is non-positive, and therefore, this term can be dropped from \eqref{vop_at} to obtain $ A_T  \, \le \, c_T A_0$. The required estimate is then obtained by inserting the elementary inequality \begin{align*}
 c_T \, &\le \, 1 - (1/2) K T^2 + (1/6)  K^2 T^4  \, \overset{\eqref{eq:CT}}{\le} \, 1 - (1/2) K T^2 + (1/48)  K T^2 (K/L)  \\
\, & \le \, 1 - (1/2 - 1/48) K T^2
\end{align*} which is valid since $L T^2 \le 1/8$ and $K \le L$.  The required result then holds because $1/2 - 1/48 = 23/48 > 1/3$.  Note, the condition $LT^2 \le 1/8$ was used twice in this proof to avoid complications due to potential periodicities in the underlying Hamiltonian flow.   
\end{proof}

\subsection{A priori upper bounds for the Stratified Monte Carlo Integrator}
The following a priori upper bounds for the sMC and semi-exact flows are useful to prove $L^2$-accuracy of the sMC flow.

\begin{lemma}[A priori bounds] \label{lem:apriori}
Suppose \ref{A1}-\ref{A2} hold. Let $T>0$ satisfy $L T^2 \le 1/8$  and let $h \ge 0$ satisfy $T/h \in \mathbb{Z}$ if $h>0$.  For any $x,y,u,v\in\R^d$, we have almost surely
	\begin{align}
	\label{apriori:1}
&\sup_{s\leq T} |\tilde{Q}_s(x,v)| \vee \sup_{s\leq T} |\bar{q}_s(x,v)|   \leq   (1+L (T^2+Th)) \max(|x|,|x+T v|) \;, \\
\label{apriori:2}
&\sup_{s\leq T}|\tilde{P}_s(x,v)| \vee \sup_{s\leq T}|\bar{p}_s(x,v)|  \leq 
|v| + L T (1+L (T^2+T h))  \max(|x|,|x+T v|). 
\end{align}
\end{lemma}
The proof of Lemma~\ref{lem:apriori} is nearly identical to the proof of Lemma~3.1 of \cite{BoEbZi2020} and hence is omitted. 

\subsection{Proof of $L^2$-Accuracy of sMC Integrator}

\label{sec:l2accuracy}

The next two lemmas combined with the triangle inequality imply $L^2$-accuracy of the sMC integrator given in Lemma~\ref{lem:smc_strong_accuracy}.

\begin{lemma}[$L^2$-accuracy of sMC Time Integrator with respect to Semi-Exact Flow]
\label{lem:smc_strong_accuracy_semi_exact}
Suppose \ref{A1}-\ref{A2} hold. Let $T>0$ satisfy $L T^2 \le 1/8$  and let $h >0$ satisfy $T/h \in \mathbb{Z}$.  Then for any $x, v \in \mathbb{R}^d$ and $k \in \mathbb{N}_0$ such that $t_k \le T$, \begin{equation}
    \left( E[ |\tilde{Q}_{t_k}(x,v) - \bar{q}_{t_k}(x,v)|^2 ] \right)^{1/2} \, \le \,  \, \sqrt{2} e^{31/8} L^{1/4} ( |v| + L^{1/2} |x| ) \, h^{3/2} \;. \label{str_err_smc_semi_exact}
\end{equation}
\end{lemma}

For a heuristic explanation of the $3/2$-order of accuracy appearing in \eqref{str_err_smc_semi_exact}, see Remark~\ref{rmk:msefundamental}.  As highlighted in the proof, this scaling results from the cancellations of $O(h)$-terms due to the independence of the random temporal sample points used by the sMC time integrator.

\begin{lemma}[Accuracy of Semi-Exact Flow]
\label{lem:semi_exact_accuracy}
Suppose \ref{A1}-\ref{A2} hold. Let $T>0$ satisfy $L T^2 \le 1/8$,  and let $h>0$ satisfy $T/h \in \mathbb{Z}$.  Then for any $x, v \in \mathbb{R}^d$ and $k \in \mathbb{N}_0$ such that $t_k \le T$, \begin{equation}
     |\bar{q}_{t_k}(x,v) - q_{t_k}(x,v)|  \, \le \,  \, 2 e^{5/8} L^{1/2}  (    |v| +   L^{1/2} |x| ) h^2 \;. \label{semi_exact_accuracy}
\end{equation}
\end{lemma}

The proofs of these lemmas use a discrete Gr\"{o}nwall inequality, which we include here for the reader's convenience.  
\begin{lemma}[Discrete Gr\"{o}nwall inequality]
\label{lem:groenwall}
Let $\lambda, h \in \mathbb{R}$ be such that $1+\lambda h > 0$.  Suppose that $(g_k)_{k \in \mathbb{N}_0}$ is a non-decreasing sequence, and $(a_k)_{k \in \mathbb{N}_0}$ satisfies $a_{k+1} \le (1 + \lambda h) a_{k} + g_{k}$ for $k \in \mathbb{N}_0$.  Then it holds \[
a_k \le (1+\lambda h)^k a_0 + \frac{1}{\lambda h} ( (1+\lambda h)^k - 1) g_{k-1} , \quad \text{for all $k \in \mathbb{N}$} \;.
\]
\end{lemma}

\begin{proof}[Proof of Lemma~\ref{lem:smc_strong_accuracy_semi_exact}]
Let $(\tilde{Q}_t(x,v), \tilde{P}_t(x,v))$ be a realization of the sMC flow from the initial condition $(x,v) \in \mathbb{R}^{2d}$ which satisfies \eqref{sMC}.   A key ingredient in this proof are the a priori upper bounds in Lemma~\ref{lem:apriori}.  In particular, since $L (T^2 + T h) \le 1/4$, which follows from the hypotheses $LT^2 \le 1/8$ and $T/h \in \mathbb{Z}$,  \eqref{apriori:2} and the Cauchy-Schwarz inequality imply that \begin{align}
\label{apriori:2a}
\sup_{s \le T} |\tilde{P}_s(x,v)|^2 \vee \sup_{s \le T} |\bar{p}_s(x,v)|^2  &\le 3 |v|^2 + 4 L^2 T^2 (|x|^2 + T^2 |v|^2) \;.
\end{align}
For all $t \ge 0$, let $\mathcal{F}_t$ denote the sigma-algebra of events up to $t$ generated by the independent sequence of random temporal sample points $(\mathcal{U}_i )_{i \in \mathbb{N}_0}$.  Define the distorted $\ell_2$- metric\footnote{This type of quadratic ``distorted'' metric naturally arises in the study of the large-time behavior of dynamical systems with a Hamiltonian part; e.g., see (2.51) of \cite{achleitner2015large}. Here we use the distorted metric differently, namely to quantify $L^2$-accuracy of the sMC time integrator.}  \begin{equation} \rho_t^2 \ := \ E \, |Z_t |^2 +  L^{-1/2} E \langle W_t, Z_t \rangle +  L^{-1} E \, |W_t |^2 \;, 
\end{equation}
where $Z_t:= \tilde{Q}_t(x,v) - \bar{q}_t(x,v)$ and $W_t:=\tilde{P}_t(x,v) - \bar{p}_t(x,v)$. 
By Young's product inequality,  \begin{equation}
\frac{1}{2} \left( E|Z_t|^2 + L^{-1} E |W_t|^2 \right)    \, \le \, \rho_t^2 \, \le \frac{3}{2} \left( E|Z_t|^2 + L^{-1} E |W_t|^2 \right)   \;. \label{ieq:rho_t}
\end{equation} As a shorthand notation, for any $k \in \mathbb{N}_0$, let  \[
\tilde{F}_{t_k} \ := \ - \nabla U(\tilde{Q}_{t_k} + (\mathcal{U}_{k} - t_k) \tilde{P}_{t_k}) \;,  \quad \text{and} \quad  \bar{F}_{t_k} \ := \   - E [\nabla U(\bar{q}_{t_k} + (\mathcal{U}_{k} - t_k) \bar{p}_{t_k})] \;.
\]   
Since the  sMC and semi-exact flows satisfy \eqref{sMC} and \eqref{semi-exact} respectively, \begin{align}
 Z_{t_{k+1}} \,  &= \, Z_{t_k} + h W_{t_k} + \frac{h^2}{2} ( \tilde{F}_{t_{k}} - \bar{F}_{t_k} )  \;,  \label{Z_tkp1} \\
W_{t_{k+1}} \, &= \,  W_{t_k} + h ( \tilde{F}_{t_{k}} - \bar{F}_{t_k} ) \;. 
\label{W_tkp1}
\end{align} 
Using, in turn, the Cauchy-Schwarz inequality, the $L$-Lipschitz continuity of $\nabla U$, and Jensen's inequality, we obtain: \begin{align}
& E | \tilde{F}_{t_{k}} - \bar{F}_{t_k} |^2 \, = \,  E \left| \frac{1}{h}  \int_{t_k}^{t_{k+1}} \left( \nabla U(\tilde{Q}_{t_k} + (\mathcal{U}_{k} - t_k) \tilde{P}_{t_k}) -\nabla U(\bar{q}_{t_k} + (u - t_k) \bar{p}_{t_k})  \right) du \right|^2 \nonumber \\ 
& \, \overset{\text{Cauchy-Schwarz}}{\le} \, E  \frac{1}{h}\int_{t_k}^{t_{k+1}}  \left| \nabla U(\tilde{Q}_{t_k} + (\mathcal{U}_{k} - t_k) \tilde{P}_{t_k}) -\nabla U(\bar{q}_{t_k} + (u - t_k) \bar{p}_{t_k})   \right|^2 du 
\nonumber \\
& \, \overset{\ref{A2}}{\le} \,   \frac{L^2}{h} E\int_{t_k}^{t_{k+1}}  \left| Z_{t_k} + (\mathcal{U}_{k} - t_k) \tilde{P}_{t_k} - (u - t_k) \bar{p}_{t_k}  \right|^2 du
\nonumber \\
& \, \overset{\text{Jensen}}{\le} \,
3 L^2  E|Z_{t_k}|^2 + 3 L^2  h^2 (E|\tilde{P}_{t_k}|^2 + |\bar{p}_{t_k}|^2 )  \label{ieq:dF2} \;.  
\end{align} 
Moreover, since the sMC flow is an unbiased estimator of the semi-exact flow, and due to independence of the random temporal sample points used by the sMC time integrator,  \begin{align}
|E  [  \tilde{F}_{t_{k}} - \bar{F}_{t_k} \mid \mathcal{F}_{t_k} ]| 
\, &= \, 
|E [ \nabla U(\bar{q}_{t_k} + (\mathcal{U}_{k} - t_k) \bar{p}_{t_k}) - \nabla U(\tilde{Q}_{t_k} + (\mathcal{U}_{k} - t_k) \tilde{P}_{t_k}) \mid \mathcal{F}_{t_k} ]|  \;, \nonumber \\
\, &\overset{\ref{A2}}{\le} \, L |Z_{t_k}| + L h |W_{t_k}| \;, \label{ieq:unbiased}
\end{align}
and it follows that, \begin{align}
E \langle Z_{t_k} , \tilde{F}_{t_{k}} - \bar{F}_{t_k} \rangle \, &= \, E \langle Z_{t_k}  , E [ \tilde{F}_{t_{k}} - \bar{F}_{t_k} \mid \mathcal{F}_{t_k} ] \rangle 
\, \le \, 
E( |Z_{t_k}|  \cdot |E [ \tilde{F}_{t_{k}} - \bar{F}_{t_k} \mid \mathcal{F}_{t_k} ]| ) \nonumber \\
& \, \overset{\eqref{ieq:unbiased}}{\le} \, L E [ |Z_{t_k}| ( |Z_{t_k}| +  h |W_{t_k}| ) ] \nonumber \\
\, &\le \, (3/2) L  E|Z_{t_k}|^2 + (1/2) L h^2  E|W_{t_k}|^2      \label{ieq:z_df}
\end{align}
where we used, in turn, the Cauchy-Schwarz inequality, \eqref{ieq:unbiased}, and Young's product inequality. Similarly,\begin{align}
E \langle W_{t_k} , \tilde{F}_{t_{k}} - \bar{F}_{t_k} \rangle  \, &\le \, L E[|W_{t_k}|( |Z_{t_k}| + h |W_{t_k}|) ] \label{ieq:w_df_raw} \\  \qquad \, &\le \, (1/2) L h^{-1} E|Z_{t_k}|^2 + (3/2) L h E|W_{t_k}|^2   \;.    \label{ieq:w_df}
\end{align}
By \eqref{Z_tkp1} and \eqref{W_tkp1},   \begin{align}
\rho_{t_{k+1}}^2 \, &= \, \rn{1}+\rn{2} +\rn{3} \qquad \text{where we have introduced} \label{eq:rn1to3} \\
    \rn{1} \, &:= \, E |Z_{t_k}|^2 + (  L^{-1/2} + 2 h ) E \langle Z_{t_k}, W_{t_k} \rangle + ( L^{-1}  +  L^{-1/2} h + h^2) E |W_{t_k}|^2   \;, \nonumber \\
   \rn{2} \, &:= \,    (h^2 +  L^{-1/2} h) E \langle Z_{t_k}, \tilde{F}_{t_{k}} - \bar{F}_{t_k} \rangle    + h (h^2 + \frac{3}{2}  L^{-1/2} h + 2 L^{-1} )  E \langle W_{t_k}, \tilde{F}_{t_{k}} - \bar{F}_{t_k} \rangle \;, \nonumber  \\
   \rn{3} \, &:= \, ( \frac{h^4}{4} +  L^{-1/2} \frac{h^3}{2} +  L^{-1} h^2 ) E |\tilde{F}_{t_{k}} - \bar{F}_{t_k}|^2 \;. \nonumber 
\end{align} Since $L h^2 \le 1/8$ implies $L^{1/2} h \le 1/2$ , \begin{align}
    \rn{1} \, & \le \, (1 + 2 L^{1/2} h ) \rho_{t_k}^2 +  (L^{1/2} h + L h^2 -2 L^{1/2} h  ) L^{-1}E |W_{t_k}|^2 \nonumber \\
    \, &\le \,  (1 + 2 L^{1/2} h ) \rho_{t_k}^2 + L^{1/2} h (1  + \frac{1}{2} -2  ) L^{-1} E |W_{t_k}|^2  \, \le \, (1 + 2 L^{1/2} h ) \rho_{t_k}^2  \label{ieq:rn1}  \;. 
\end{align}
By \eqref{ieq:z_df}, \eqref{ieq:w_df_raw} and \eqref{ieq:w_df},  \begin{align}
    \rn{2} \, &\le \,  (h^2 +  L^{-1/2} h) \left( \frac{3}{2} L E|Z_{t_k}|^2 + \frac{1}{2} L h^2  E|W_{t_k}|^2   \right) \nonumber \\ 
    & ~+ (h^2 + \frac{3}{2}  L^{-1/2} h  ) \left( \frac{1}{2} L E|Z_{t_k}|^2 + \frac{3}{2} L h^2 E|W_{t_k}|^2 \right) +  2hL^{-1} E \langle W_{t_k}, \tilde{F}_{t_{k}} - \bar{F}_{t_k} \rangle  \nonumber \\
    & \overset{\eqref{ieq:w_df_raw}}{\le} ( 2 L h^2 + \frac{9}{4}  L^{1/2} h   ) E|Z_{t_k}|^2 + ( 2 (L h^2)^2 + \frac{11}{4}  (L h^2)^{3/2} ) L^{-1} E|W_{t_k}|^2  \nonumber \\ 
    & ~+  2h E [|W_{t_k}| (|Z_{t_k}| + h |W_{t_k}|)] \nonumber\\
    & \le ( 2 L h^2 + \frac{9}{4}  L^{1/2} h   ) E|Z_{t_k}|^2 + ( 2 (L h^2)^2 + \frac{11}{4}  (L h^2)^{3/2} ) L^{-1} E|W_{t_k}|^2  \nonumber\\ 
    &~ + 2(Lh^2)L^{-1} E[|W_{t_k}|^2] + 2h \left( \frac{L^{1/2}}{2} E[|Z_{t_k}|^2] + \frac{1}{2L^{1/2}} E[|W_{t_k}|^2] \right)  \nonumber \\
    & = \left( \frac{13}{4} + 2 L^{1/2}h \right) L^{1/2} h E[|Z_{t_k}|^2] \nonumber\\ 
    &~ + \left( 2 (Lh^2)^{3/2} + \frac{11}{4} Lh^2 + 2 L^{1/2}h + 1 \right) (L^{1/2} h) L^{-1} E[|W_{t_k}|^2] \nonumber \\
    & \le \left( \frac{13}{4}  + 1 \right) L^{1/2} h E [ |Z_{t_k}|^2] + \left( \frac{1}{8} + \frac{11}{32} + 2 \right) (L^{1/2} h) L^{-1}E[ |W_{t_k}|^2] \nonumber \\
    & \le \frac{17}{4} L^{1/2}h  \left( E[|Z_{t_k}|^2] + L^{-1} E[|W_{t_k}|^2\right)  \, \overset{\eqref{ieq:rho_t}}{\leq} \, \frac{17}{2} (L^{1/2} h) \rho_{t_k}^2  \label{ieq:rn2}
    \end{align} 
    where we used $L^{1/2} h \le 1/2$, $L h^2 \le 1/8$ and \eqref{ieq:rho_t}. Finally, by \eqref{ieq:dF2}, and using once more $L^{1/2} h \le 1/2$, $L h^2 \le 1/8$ and \eqref{ieq:rho_t}, 
\begin{align}
    \rn{3} \, &\le \,  ( \frac{h^4}{4} +  L^{-1/2} \frac{h^3}{2} +  L^{-1} h^2 )  \left( 3 L^2  E|Z_{t_k}|^2 + 3 L^2  h^2 (E|\tilde{P}_{t_k}|^2 + |\bar{p}_{t_k}|^2 )   \right) \nonumber \\
    \, &\le \, 2 L^{1/2} h E|Z_{t_k}|^2  + ( \frac{3}{4} (L h^2)^2 + \frac{3 }{2} (L h^2)^{3/2} + 3  L h^2  ) 
    h^2 (E|\tilde{P}_{t_k}|^2 + |\bar{p}_{t_k}|^2 ) \nonumber \\
    \, &\overset{\eqref{ieq:rho_t}}{\leq} \, 4 L^{1/2} h \rho_{t_k}^2 + ( \frac{3}{4} (L h^2)^2 + \frac{3 }{2} (L h^2)^{3/2} + 3  L h^2  ) 
    h^2 ( L |x|^2 + \frac{49}{8} |v|^2 ) \label{ieq:rn3} 
    \end{align}
    where in the last step we inserted the a priori upper bounds from \eqref{apriori:2a}.  Inserting \eqref{ieq:rn1}, \eqref{ieq:rn2} and \eqref{ieq:rn3}  into \eqref{eq:rn1to3} yields \begin{align*}
        \rho_{t_{k+1}}^2 \le (1 + \frac{29}{2} L^{1/2} h) \rho_{t_k}^2 +  ( \frac{3}{4} (L h^2)^2 + \frac{3 }{2} (L h^2)^{3/2} + 3  L h^2  ) 
    h^2 ( L |x|^2 + \frac{49}{8} |v|^2 ) \;.
    \end{align*}  
By discrete Gr\"{o}nwall inequality (Lemma~\ref{lem:groenwall}), 
\begin{align*}
& \rho_{t_{k}}^2 \le  \frac{6}{29} e^{29 L^{1/2} T /2}  ( \frac{1}{4} (L h^2) + \frac{1}{2} (L h^2)^{1/2} + 1 ) 
    (L h^2)^{1/2} h^2 ( L |x|^2 + \frac{49}{8} |v|^2) \\ 
& \quad  \le \,  \left(\frac{6}{29} \right) \left( \frac{41}{32} \right) \left( \frac{49}{8}\right) e^{29/4} L^{1/2} h^3 (L|x|^2 + |v|^2)  \le e^{31/4} L^{1/2}h^3 (L|x|^2 + |v|^2)\;.
\end{align*}
Here we simplified via $L T^2 \le 1/8$, $L^{1/2} T \le 1/2$ and $T/h \in \mathbb{Z}$. Employing  \eqref{ieq:rho_t} gives the required upper bound.  
\end{proof}

\begin{proof}[Proof of Lemma~\ref{lem:semi_exact_accuracy}]
Define the weighted  $\ell_1$-metric  \begin{equation} \label{eq:ell1_rho_t}
\rho_t := |z_t| + L^{-1/2} |w_t| \;, 
\end{equation}
where $z_t := \bar{q}_t(x,v) - q_t(x,v)$ and $w_t := \bar{p}_t(x,v) - p_t(x,v)$.  As a shorthand, let $\bar F_{t_k} := -E [ \nabla U( \bar{q}_{t_k} + (\mathcal{U}_k - t_k) \bar{p}_{t_k} ) ]$ and $F_t := - \nabla U(q_t)$.  
Since the exact and semi-exact flows satisfy \eqref{eq:Hamflow} and \eqref{semi-exact} respectively, we have \begin{align*}
    z_{t_{k+1}} &= z_{t_k} + h w_{t_k} + \int_{t_k}^{t_{k+1}} \int_{t_k}^s [ \bar{F}_{t_k} - F_r ] dr ds \;,   \\
    w_{t_{k+1}} &= w_{t_k} + h \bar{F}_{t_k} - \int_{t_k}^{t_{k+1}} F_s ds \;. 
\end{align*}
By the triangle inequality,
\begin{align}
    |z_{t_{k+1}}| &\le |z_{t_k}| + h |w_{t_k}| + \int_{t_k}^{t_{k+1}} \int_{t_k}^s | \bar{F}_{t_k} - F_r| dr ds \;,  \label{z_tkp1} \\
    |w_{t_{k+1}}| &\le |w_{t_k}| + |h \bar{F}_{t_k} - \int_{t_k}^{t_{k+1}} F_s ds| \;. \label{w_tkp1} 
\end{align}
By the triangle inequality and \ref{A2}, \begin{align}
& | \bar{F}_{t_k} - F_r | \, \le \, | F_r - F_{t_k} | + | F_{t_k} - \bar{F}_{t_k} | = | \nabla U(q_r) - \nabla U(q_{t_k}) | + | F_{t_k} - \bar{F}_{t_k} |  \nonumber \\
& \quad \overset{\ref{A2}}{\le}  L |\int_{t_k}^r p_u du | + L ( |z_{t_k}| + h \sup_{s \le T} |\bar{p}_{s}| ) \nonumber \\
& \quad \le  L h \sup_{s \le T} |p_s| + L ( |z_{t_k} | + h \sup_{s \le T} |\bar{p}_{s}| ) \;.   \label{eq:dbarf1}  
\end{align}
Moreover, since the semi-exact flow incorporates the average of the potential force over each stratum,  \begin{align}
& | h \bar{F}_{t_k} - \int_{t_k}^{t_{k+1}} F_s ds |  \\
& \qquad \, \le \,  \left|\int\limits_{t_k}^{t_{k+1}} \Big[ \nabla U\Big(\bar{q}_{t_k} + (s- t_k) \bar{p}_{t_k}\Big) - \nabla U\Big(q_{t_k} + (s-t_k) p_{t_k} + \int\limits_{t_k}^s (s-r) F_r dr\Big) \Big] ds \right| \nonumber \\
& \qquad  \overset{\ref{A2}}{\le}   L ( h |z_{t_k}| + \frac{h^2}{2} |w_{t_k}| + \frac{h^3}{6} \sup_{s \le T} |  \nabla U(q_s) | ) \nonumber \\
& \qquad  \overset{\ref{A2}}{\le}   L \left( h |z_{t_k}| + \frac{h^2}{2} |w_{t_k}| + \frac{h^3}{6} ( L |x| + L T \sup_{s \le T} |p_s| ) \right) \;, \label{eq:dbarf2}
\end{align}
where in the last step we used 
\begin{align*}
\sup_{s \le T} |\nabla U(q_s)| \ &= \ \sup_{s \le T}  |\nabla U(x) + \nabla U(q_s) - \nabla U(x)|  \\
& \overset{\ref{A1},\ref{A2}}{\le}  L \, |x| + L \, \sup_{s \le T} |\int_0^s p_u du| \, \le \, L \, |x| + L T \, \sup_{s \le T} |p_s| \;. 
\end{align*}
Inserting \eqref{eq:dbarf1} and \eqref{eq:dbarf2} into  \eqref{z_tkp1} and \eqref{w_tkp1} respectively, yields \begin{align}
    |z_{t_{k+1}}| \, &\le \, (1+ \frac{1}{2} L h^2) |z_{t_k}| + h |w_{t_k}| +  L h^3 (\sup_{s \le T} |p_s| \vee \sup_{s \le T} |\bar{p}_s| )   \label{ieq:z_tkp1} \;, \\
    |w_{t_{k+1}}| \, &\le \, (1+ \frac{1}{2} L h^2) |w_{t_k}| + L h |z_{t_k}| + \frac{1}{6} L^2 h^3 (  |x| +  T \sup_{s \le T} |p_s| ) \;. 
    \label{ieq:w_tkp1}
\end{align}
Inserting \eqref{ieq:z_tkp1} and \eqref{ieq:w_tkp1} into \eqref{eq:ell1_rho_t} evaluated at $t=t_{k+1}$, and using $L^{1/2} h \le 1/2$, gives \begin{align}
  &  \rho_{t_{k+1}} \, \le \, (1+ \frac{5}{4} L^{1/2} h) \rho_{t_k} + \frac{1}{6} (L^{1/2} h)^3 |x|  + ( L^{1/2}h )^3 (  L^{-1/2} + \frac{1}{6} T )  ( \sup_{s \le T} |p_s| \vee \sup_{s \le T} |\bar{p}_s|  ) \;. \nonumber
\end{align}
By the discrete Gr\"{o}nwall's inequality (Lemma~\ref{lem:groenwall}), \begin{align}
 &   \rho_{t_{k}} \, \le \,  \frac{4}{5} e^{(5/4) L^{1/2} T} \left( \frac{1}{6} (L^{1/2} h)^2 |x| + ( L^{1/2}h )^2 (  L^{-1/2} + \frac{1}{6} T )  ( \sup_{s \le T} |p_s| \vee \sup_{s \le T} |\bar{p}_s|  ) \right)   \nonumber \\
& \quad \, \le \,    \frac{4}{5} e^{5/8} L^{1/2} h^2 \left( \frac{1}{6} L^{1/2} |x| + (1 + \frac{1}{6} L^{1/2} T)  ( \sup_{s \le T} |p_s| \vee \sup_{s \le T} |\bar{p}_s|  ) \right) \nonumber \\
& \quad \, \le \frac{4}{5} e^{5/8} L^{1/2}h^2 \left(\frac{1}{6} L^{1/2} |x| + \left( \frac{13}{12}\right) \left(\frac{5}{8} L^{1/2}|x| + \frac{37}{32} |v|\right) \right) \nonumber\\
& \quad \, \le \, \frac{4}{5} e^{5/8} L^{1/2} h^2 ( \frac{27}{32}  L^{1/2} |x| + \frac{481}{384} |v| )  
 \, \le \, 2 e^{5/8} L^{1/2}  (L^{1/2} |x| + |v|) h^2 \nonumber \; ,
\end{align}
which gives \eqref{semi_exact_accuracy} --- as required.  Note that in the last two steps we inserted the \emph{a priori} upper bound in \eqref{apriori:2} and applied the conditions $L T^2 \le 1/8$  and $L^{1/2} h \leq 1/2$.
\end{proof}

\pagebreak

\appendix

\section{Metropolis-Adjustable Randomized Time Integrators}

\label{sec:aHMC}

Here we present Metropolis-adjustable randomized time integrators, which are `adjustable' in the following sense.  

\begin{definition} \label{defn:adjustable}
Let $\mu_{BG} := \mu \otimes \mathcal{N}(0, I_d)$ be the Boltzmann-Gibbs distribution.
A randomized time integrator with  transition kernel $\tilde{\Pi}$ is \emph{adjustable} if there exists a  density $g$ on $\mathbb{R}^{4 d}$ such that: \[
\mu_{BG} (dq \, dv) \ \tilde{\Pi}( \mathcal{S}(q,v), \mathcal{S}(dq' \, dv') )   \ = \ g((q,v), (q',v')) \ \mu_{BG} (dq' \, dv') \,  \tilde{\Pi}( (q',v'), dq \, dv ) 
\] where $\mathcal{S}: (q,v) \mapsto (q,-v)$ is the velocity flip map.  
\end{definition}

In particular, the transition kernel $\tilde{\Pi}$ of an adjustable randomized time integrator can be used as a proposal distribution in a generalized Metropolis-Hastings algorithm with target measure $\mu_{BG}$ and transition kernel \[
\Pi((q,v), (dq' dv')) \ = \ \alpha ((q,v), (q',v')) \ \tilde{\Pi}((q,v), dq' dv') + r(q,v) \ \delta_{\mathcal{S}(q,v)}(dq' dv') 
\] where $\alpha((q,v), (q',v')) := \min(1, g((q',v'), (q,v)) )$.  It is straightforward to verify that the corresponding Metropolis-adjusted kernel $\Pi$ satisfies generalized detailed balance w.r.t.~$\mu_{BG}$, i.e., \[
\mu_{BG} (dq \, dv) \ \Pi( \mathcal{S}(q,v), \mathcal{S}(dq' \, dv') )   \ = \  \mu_{BG} (dq' \, dv') \  \Pi( (q',v'), dq \, dv ) 
\] and hence, leaves $\mu_{BG}$ invariant.  Since the exact Hamiltonian flow $\varphi_t(x,v):= (q_t(x,v), p_t(x,v))$ in \eqref{eq:Hamflow} is $\mathcal{S}$-reversible (i.e., $\varphi_{-t} = \mathcal{S} \circ \varphi_t \circ \mathcal{S}$) \cite{BoSaActaN2018}, the corresponding transition kernel $\Pi_{t}((x,v), \cdot) := \delta_{\varphi_t(x,v)}$ satisfies generalized detailed balance w.r.t.~$\mu_{BG}$.  A general class of adjustable randomized time integrators are provided by the following proposition. 

\begin{proposition}
Let $\{ \theta_i \}_{i \in \mathcal{I}}$ be an indexed family of time integrators for \eqref{eq:Hamflow} with index set $\mathcal{I}$.  Let $\rho$ be a probability distribution over $\mathcal{I}$.  For each $i \in \mathcal{I}$, suppose that $\theta_i$ is: (i) volume-preserving and (ii) $\mathcal{S}$-reversible.  Define     \[ \theta_{u_N\,.\,.\,u_1} \ \equiv \ \theta_{u_N} \circ \cdots \circ \theta_{u_2} \circ \theta_{u_1} \;, 
 \quad \text{where $u_j \in \mathcal{I}$ for each $j \in \{1, \dots, N \}$} \;.
  \]  Then the randomized time integrator with transition kernel 
  \[
  \tilde{\Pi}((q,v), \cdot) \ = \int_{\mathcal{I}^N} \delta_{ \theta_{u_N\,.\,.\,u_1}(q,v)} \prod_{i=1}^N \rho(d u_i)
  \] is adjustable with $g((q,v), (q',v')) = e^{-[H(q,v)-H(q',v')]}$.
\end{proposition}

The proof of this proposition is straightforward and therefore omitted.  As an application of this proposition, here is a concrete, implementable example of an adjustable randomized time integrator.  

\begin{example}[Adjustable Randomized Time Integrator]
\label{eg:2stage} 
Let $\mathcal{I} = [0,  1/2]$, $h >0$ be a time step size, $b \in \mathcal{I}$, and $\theta_b: \mathbb{R}^{2d} \to \mathbb{R}^{2d}$  be the following 2-stage palindromic integrator \begin{align}
&\theta_b(q,v) \ := \ \varphi^{(A)}_{b h} \circ \varphi^{(B)}_{h/2}  \circ \varphi^{(A)}_{(1-2 b) h} \circ  \varphi^{(B)}_{h/2} \circ \varphi^{(A)}_{b h} (q,v)  \;, \label{eq:2-stage}  \\
&\text{where} \quad  \varphi_t^{(A)}(x,v) \ := \ (x + t v, v) \;,~~  \text{and} \quad \varphi_t^{(B)}(x,v) \ := \ (x, v + t F(x)) \;. \nonumber  
\end{align}
More explicitly, the scheme can be written as \begin{align*}
 (x,v)\mapsto \left( x + h v + (1-b) \frac{h^2}{2} F_+ + b \frac{h^2}{2}  F_-,~ v + \frac{h}{2} [ F_+ + F_- ] \right) \;,
\end{align*}
where  $F_+ :=  F(x+b h v)$ and $F_-  :=   F\big(x+(1-b) h v + (1-2 b) \frac{h^2}{2} F_+ \big)$.
\end{example}

In general, this adjustable randomized time integrator requires two potential force evaluations per integration step.  However, for $b=0$ and $b=1/2$, the scheme reduces to the velocity and position Verlet schemes, respectively. In this example, the probability distribution $\rho$ could be  the uniform distribution over $\operatorname{Uniform}(0,1/2)$.  Alternatively, in order to average over position and velocity Verlet, $\rho$ could be $\operatorname{Uniform}\{0, 1/2 \}$.

\begin{algorithm3}[Adjusted HMC with Randomized Time Integrator]
\label{algo:adjusted_hmc}
Given \# of integration steps $N \in \mathbb{N}$, a probability measure $\rho$ on $\mathcal{I}$,  and the current state $(  Q_0^a,  V_0^a ) \in \mathbb{R}^{2d}$, the method outputs an updated state  $(  Q_1^a,  V_1^a ) \in \mathbb{R}^{2d}$ using:
\begin{description}
\item[Step 1] Draw independently $\{ \mathcal{U}_i \}_{i=1}^N \overset{i.i.d.}{\sim} \rho$, $\mathcal{V} \sim \operatorname{Uniform}(0,1)$, and $\xi \sim \mathcal{N}(0,I_d)$.
\item[Step 2] Set 
\begin{equation*}
 \hspace{-0.25in}
 (  Q_1^a,  V_1^a ) = \begin{cases}\theta_{\mathcal{U}_N\,.\,.\,\mathcal{U}_1}( Q_0^a, \xi )  & \text{if $\mathcal{V} \le \exp(-[H(\theta_{\mathcal{U}_N\,.\,.\,\mathcal{U}_1}( Q_0^a, \xi ) )-H(Q_0^a, \xi )]^+)$} \;, \\
 \mathcal{S}(Q_0^a, \xi ) & \text{otherwise} \;.
\end{cases} 
\end{equation*}
\end{description}
\end{algorithm3}

\section{Duration-Randomized uHMC with sMC Time Integration}

\label{sec:dr_uHMC}

Here we consider a duration-randomized uHMC algorithm with complete velocity refreshment (or \emph{randomized uHMC} for short).  In order to avoid periodicities in the Hamiltonian steps of HMC, duration randomization was suggested by Mackenzie in 1989 \cite{Ma1989}. 
There are a number of ways to incorporate duration randomization into uHMC \cite{CaLeSt2007,Ne2011, BoSa2017, BoSaActaN2018}. One overlooked way, which is perhaps the simplest to analyze, is the unadjusted Markov jump process on phase space introduced in \cite[Section 6]{BoSa2017}, as briefly recounted below.   

Before describing this process, we note that duration randomization has a similar effect as  partial velocity refreshment.  Intuitively speaking, after a random number of non-randomized uHMC transition steps with partial velocity refreshment, a complete refreshment occurs.  Therefore, the findings given below are expected to hold for  uHMC with partial velocity refreshment. However, in comparison to randomized uHMC, the analysis of  uHMC with partial velocity refreshment is a more  demanding task if the bounds have to be realistic with respect to model/hyper parameters.

\begin{comment}
\subsection{Notation}

While uHMC was defined on $\mathbb{R}^d$, duration-randomized uHMC is defined on the enlarged \emph{phase space} $\mathbb{R}^{2 d}$.  Let $\PM(\mathbb{R}^{2d})$ denote the set of all probability measures on phase space $\mathbb{R}^{2d}$, and denote by  $\PM^p(\mathbb{R}^{2d})$ the subset of probability measures on $\mathbb{R}^{2d}$ with finite $p$-th moment.    Denote the set of all couplings of $\nu, \eta \in \PM(\mathbb{R}^{2d})$ by $\J(\nu,\eta)$. 
For $\nu,\eta \in \PM(\mathbb{R}^{2d})$,
define the $L^p$-Wasserstein distance w.r.t.\ the metric $\mathrm{d}$ by
\[
\W^p_{\mathrm{d}}(\nu,\eta) \ := \ \left( \inf \Big\{ E\left[(\mathrm{d}(X,Y)^p  \right] ~:~ \law(X, Y) \in \J(\nu,\eta) \Big\} \right)^{1/p} \;.
\] 

\end{comment}

\subsection{Definition of Randomized uHMC with sMC time integration}

The randomized uHMC process is an implementable, inexact MCMC method defined on \emph{phase space} $\mathbb{R}^{2d}$ and aimed at the Boltzmann-Gibbs distribution   \begin{equation}
\mu_{BG} := \mu \otimes \mathcal{N}(0, I_d) \;. 
\end{equation} 
%Moreover, under our assumptions on $U$, it can be verified that $\mu_{BG}$ is also the unique invariant measure of the exact process.
First, we define the infinitesimal generator of the randomized uHMC process; and then describe how a path of this process can be realized.  

\smallskip

To define the infinitesimal generator, let $\mathcal{U} \sim \operatorname{Uniform}(0,h)$ and $\xi \sim \mathcal{N}(0,I_d)$  be independent random variables.   Denote by $\Theta_h(x,v,\mathcal{U})$ a single step of the sMC time integrator operated with time step size $h>0$ and initial condition $(x,v) \in \mathbb{R}^{2 d}$ where $\Theta_h: \mathbb{R}^{2 d} \times (0,1) \to \mathbb{R}^{2 d}$ is the deterministic map defined by \begin{equation} \label{eq:Theta_h}
\Theta_h : (x,v,u) \mapsto \Big( x + h v + \frac{h^2}{2} F(x+u v), v + h F(x+u v) \Big) \;. 
\end{equation} 
   On functions $f:\mathbb{R}^{2 d} \to \mathbb{R}^{2 d}$, the infinitesimal generator of the randomized uHMC process is defined by
\begin{equation}\label{eq:generator1}
 \tilde{\mathcal{G}} f(x,v) = 
  h^{-1} E \left(  f(\Theta_h(x,v,\mathcal{U})) - f(x,v) \right)  +
 \lambda E \left( f( x,\xi ) - f(x,v) \right)  \;, 
\end{equation}
where $\lambda>0$ is the intensity of velocity randomizations and $h>0$ is the step size.  The operator $\tilde{\mathcal{G}}$ is the generator of a Markov jump process $(\tilde{Q}_t, \tilde{P}_t)_{t \ge 0}$  with jumps that result in either: (i) a step of the sMC time integrator $\Theta_h$; or (ii) a complete velocity refreshment $(x,v) \mapsto (x,\xi)$.  Due to time-discretization error in the sMC steps, this process has an asymptotic bias.  Since the number of jumps of the process over $[0,t]$ is a Poisson process with intensity $\lambda + h^{-1}$, the mean number of steps  of $\Theta_h$ (and hence, gradient evaluations) over a time interval of length $t>0$ is $t/h$.    

\smallskip

The random jump times and embedded chain of the randomized uHMC process may be produced by iterating the following algorithm.  

\smallskip

\begin{algorithm3*}[Randomized uHMC]
Given intensity $\lambda > 0 $, step size $h>0$, the current time $T_0$, and the current state $( \tilde Q_{T_0}, \tilde V_{T_0} ) \in \mathbb{R}^{2d}$, the method outputs an updated state  $( \tilde Q_{T_1}, \tilde V_{T_1} ) \in \mathbb{R}^{2d}$ at the random time $T_1$ using:
\begin{description}
\item[Step 1] Draw an exponential random variable $\Delta T$ with mean $h/(\lambda h + 1)$, and update time via $T_1 = T_0 + \Delta T$.
\item[Step 2] Draw $\xi \sim \mathcal{N}(0,I_d)$, $\mathcal{U} \sim \operatorname{Uniform}(0,h)$, $\mathcal{V} \sim \operatorname{Uniform}(0,1)$, and set 
\begin{align*}
& ( \tilde Q_{T_1}, \tilde V_{T_1} ) = \begin{dcases}
(\tilde Q_{T_0}, \xi ) & \mathcal{V} \leq \frac{\lambda h}{1+ \lambda h}\\
\Theta_h(\tilde Q_{T_0}, \tilde V_{T_0}, \mathcal{U} ) &  \text{otherwise} \;.
\end{dcases}
\end{align*}
\end{description}
Note: the random variables $\Delta T$, $\xi$, $\mathcal{V}$, and $\mathcal{U}$ are mutually independent and independent of the state of the process.  
\end{algorithm3*}

Let $\{ T_i \}_{i \in \mathbb{N}_0}$ and $\{ (\tilde Q_{T_i}, \tilde V_{T_i} ) \}_{i \in \mathbb{N}_0}$ denote the sequence of random jump times  and states  obtained by iterating this algorithm respectively.
The path of the randomized uHMC process is then given by \[
(Q_t, p_t) = ( \tilde Q_{T_i}, \tilde V_{T_i} )  \quad \text{for  $t \in [T_i, T_{i+1})$} \;.
\]
Moreover, for any $t>0$, the time-average of an observable $f: \mathbb{R}^{2d} \to \mathbb{R}$ along this trajectory is  given by: \[
\frac{1}{t} \int_0^t f(Q_s, p_s) ds = \frac{1}{t} \sum_{0 \le i \le \infty} f(Q_{T_i}, V_{T_i} ) (t \wedge T_{i+1} - t \wedge T_i )\;.
\]
%Moreover, the mean holding time of this Markov jump process at any state is constant and given by E \delta t = \frac{h}{\lambda h + 1} \;. \end{equation}

Let $\theta_h: \mathbb{R}^{2 d} \times (0,1) \to \mathbb{R}^{2 d}$ denote the map that advances the exact solution of the Hamiltonian dynamics over a single time step of size $h>0$, i.e., \[
\theta_h: (x,v) \mapsto \Big( x + h v + \int_0^h (h-s) F(q_s(x,v)), v + \int_0^h F(q_s(x,v))  ds \Big) \;.
\]  In the asymptotic bias proof, we couple the randomized uHMC process to a corresponding \emph{exact}  process $(Q_t, p_t)_{t \ge 0}$ with generator defined by:  \begin{equation} 
\begin{aligned}
& \mathcal{G} f(x,v) =   h^{-1} E \left(  f(\theta_h(x,v)) - f(x,v) \right)  \\
& \qquad + \lambda E \left( f( x,\xi ) - f(x,v) \right)  \;. 
\end{aligned}
\end{equation} A key property of the exact process is that it leaves infinitesimally invariant the Boltzmann-Gibbs distribution $\mu_{BG}$, and under our regularity assumptions on the target measure $\mu$, it can be verified that $\mu_{BG}$ is also the unique invariant measure of the exact process.

\subsection{$L^2$-Wasserstein Contractivity of Randomized uHMC}

Let $(\tilde{p}_t)_{t \ge 0}$ denote the transition semigroup of the randomized uHMC process $((\tilde{Q}_t, \tilde{P}_t))_{t \ge 0}$.  A key tool in the contraction proof is a coupling of two copies of the randomized uHMC process with generator defined by \begin{equation} 
\label{eq:GC_contr}
\begin{aligned}
& \tilde{\mathcal{G}}^C f(y) \, = \,  h^{-1} E \left(  f(\Theta_h(x,v,\mathcal{U}),\Theta_h(\tilde{x},\tilde{v},\mathcal{U})) - f(y) \right)  \\
 & \qquad +  \lambda E \left( f( (x,\xi), (\tilde{x}, \xi) ) - f( y ) \right)
 \end{aligned}
\end{equation}
where $y = ((x,v), (\tilde{x}, \tilde{v})) \in \mathbb{R}^{4 d}$. 
This coupling proof parallels the coupling approach in Ref.~\cite{BoEb2022} for the PDMP corresponding to duration-randomized \emph{exact} HMC; another strategy to prove contractivity is based on hypocoercivity \cite{lu2022explicit}.   

\medskip

\smallskip

To measure the distance between the two copies we use a distorted metric: \begin{equation}
\label{eq:dist_metric}
\begin{aligned}
& \rho( y )^2  :=  \frac{1}{4} |z|^2 + \frac{\lambda^{-1}}{2}   \langle z, w \rangle + \lambda^{-2} |w|^2  =  \begin{pmatrix} z & w \end{pmatrix} \mathsf{A} \begin{pmatrix} z \\ w \end{pmatrix} \;, \quad \text{where} \\
& y = ((x,v), (\tilde{x}, \tilde{v})) \in \mathbb{R}^{4 d} \;, ~~
z := x- \tilde{x} \;, ~~ w := v- \tilde{v} \;, ~~
\mathsf{A}  :=  \begin{bmatrix} \frac{1}{4} \mathbf{1}_{d} & \frac{\lambda^{-1}}{4} \mathbf{1}_{d} \\
\frac{\lambda^{-1}}{4}  \mathbf{1}_{d} & \lambda^{-2} \mathbf{1}_{d} \end{bmatrix} \;. 
\end{aligned}
\end{equation}
This distorted metric involves the ``$qv$ trick'' behind Foster-Lyapunov functions for (i) dissipative Hamiltonian systems with random impulses \cite{SaSt1999}; (ii) second-order Langevin processes \cite{wu2001large,MaStHi2002,Ta2002,achleitner2015large}; and (iii) exact randomized HMC \cite{BoSa2017}.  This cross-term plays a crucial role since it captures the contractivity of the potential force.  Using the Peter-Paul inequality with parameter $\delta$, we can compare this distorted metric to a `straightened' metric, \begin{align}
\rho(y)^2 \, &\le \, \left( \frac{1}{4} + \frac{ \lambda^{-1} \delta}{4} \right) |z|^2 + \left( \lambda^{-2} + \frac{\lambda^{-1}}{4 \delta}  \right) |w|^2 \, \overset{\delta=3\lambda}{\le} \, |z|^2 + \frac{13}{12} \lambda^{-2}   |w|^2  \label{eq:straight_metric0} \\
\, &\le \, \max(K^{-1}, \frac{13}{12} \lambda^{-2}) \left( K |z|^2 + |w|^2  \right) \;. \label{eq:straight_metric}
\end{align}  Similarly, the distorted metric is equivalent to the standard Euclidean metric \begin{equation} \label{eq:equiv_dist_metric}
\frac{1}{8} \min(1, 4 \lambda^{-2}) \left( |z|^2 +  |w|^2 \right) \, \le \, \rho( y )^2 \, \le \,   \max(1, \frac{13}{12} \lambda^{-2}) \left( |z|^2 +  |w|^2 \right) \;.
\end{equation}  

\smallskip

By applying the generator $\tilde{\mathcal{G}}^C $ on this distorted metric, and using the co-coercivity property of $\nabla U$ (see Remark~\ref{rmk:cocoercivity}), we can prove the following.

\smallskip

\begin{lemma}
\label{lem:inf_contr_mjp}
Suppose that Assumptions~\ref{A1}-\ref{A3} hold and $\lambda>0$, $h>0$ satisfy \begin{align}
(L \lambda^{-2} )^{1/2} \, &\le \, 12^{-1}  \label{eq:Cl} \;, \\
\lambda h \, &\le \,  1 \;. \label{eq:Ch}
\end{align}
Then $\tilde{\mathcal{G}}^C$ satisfies the following infinitesimal contractivity result 
\begin{equation} \label{eq:inf_contr_mjp}
 \tilde{\mathcal{G}}^C \rho^2 \, \le \, - \, \gamma \, \rho^2 \;,  \quad \text{where $\gamma \, := \, 10^{-1} \frac{K}{\lambda} $}  \;. 
\end{equation}
\end{lemma}

The proof of Lemma~\ref{lem:inf_contr_mjp} is deferred to Section~\ref{sec:randomized_uHMC_proofs}.  As a consequence of Lemma~\ref{lem:inf_contr_mjp}, we can prove $L^2$ Wasserstein contractivity of the transition semigroup $(\tilde{p}_t)_{t \ge 0}$ of randomized uHMC.  

\begin{theorem}
\label{thm:contr_mjp}
Suppose that Assumptions~\ref{A1}-\ref{A3} hold and $\lambda>0$, $h>0$ satisfy \eqref{eq:Cl} and \eqref{eq:Ch}, respectively.  Then
 for any pair of probability measures $\nu, \eta \in \PM^2(\mathbb{R}^{2d})$, and for any $t \ge 0$, \begin{equation}
\W^2(\nu \tilde{p}_t, \eta \tilde{p}_t) \, \le \, 3 \, \max(\lambda, \lambda^{-1}) \, e^{-\gamma \, t \, /2 } \, \W^2(\nu, \eta) \;.
\end{equation}
\end{theorem}

\begin{proof}[Proof of Theorem~\ref{thm:contr_mjp}]
Let $(Y_t)_{t \ge 0}$ denote the coupling process  on $\mathbb{R}^{4 d}$ generated by $ \tilde{\mathcal{G}}^C $ with initial distribution given by an optimal coupling of the initial distributions $\nu$ and $\eta$ w.r.t.\ the distance $\W^2$.  
As a consequence of Lemma~\ref{lem:inf_contr_mjp}, the process $t \mapsto e^{\gamma t} \rho(Y_t)^2$ is a non-negative supermartingale.  Moreover, by using the equivalence to the standard Euclidean metric given in \eqref{eq:equiv_dist_metric}, \begin{align*}
\W^2(\nu \tilde{p}_t, \eta \tilde{p}_t)^2 \, &\le \, 8 \max(1, \frac{1}{4} \lambda^{2}) E[ \rho(Y_t)^2 ] \, \overset{\eqref{eq:inf_contr_mjp}}{\le} 8 \max(1, \frac{1}{4} \lambda^{2}) e^{- \gamma t} E[ \rho(Y_0)^2 ] \\ 
\, &\le \, 8 \, \max(1, \frac{1}{4} \lambda^{2}) \, \max(1, \frac{13}{12} \lambda^{-2} ) \, e^{-\gamma t} \, \W^2(\nu, \eta)^2 \\ 
\, &\le \, 9 \, \max( \lambda^2, \lambda^{-2} ) \, e^{-\gamma t} \, \W^2(\nu, \eta)^2 \;.
\end{align*}  Taking square roots of both sides gives the required result.  
\end{proof}

\subsection{$L^2$-Wasserstein Asymptotic Bias of Randomized uHMC}

As a consequence of Theorem~\ref{thm:contr_mjp}, the randomized uHMC process admits a unique invariant measure denoted by $\tilde{\mu}_{BG}$.  Here we quantify the $L^2$-Wasserstein asymptotic bias, i.e., $\W^2(\mu_{BG}, \tilde{\mu}_{BG})$.  
 A key tool in the asymptotic bias proof is a coupling of the unadjusted and exact processes with generator \begin{equation}
\label{eq:AC_asympbias}
\begin{aligned}
& \mathcal{A}^C f(y) \, = \,  h^{-1} E \left(  f(\theta_h(x,v),\Theta_h(\tilde{x},\tilde{v},\mathcal{U})) - f(y) \right)  \\
& \qquad + \lambda E \left( f( (x,\xi), (\tilde{x}, \xi) ) - f(y) \right) \;,
 \end{aligned}
\end{equation}
where $y = ((x,v), (\tilde{x}, \tilde{v})) \in \mathbb{R}^{4 d}$.  

\begin{lemma}
\label{lem:inf_drif_mjp}
Suppose that Assumptions~\ref{A1}-\ref{A3} hold and $\lambda>0$ and $h>0$ satisfy \eqref{eq:Cl} and \eqref{eq:Ch}.  Let $\gamma$ be the contraction rate of randomized uHMC in \eqref{eq:inf_contr_mjp}.
Then $\mathcal{A}^C$ satisfies the following infinitesimal drift condition  \begin{equation} \label{eq:inf_drif_mjp}
\mathcal{A}^C \rho(y)^2  \, \le \, - \, \frac{\gamma}{2} \, \rho(y)^2 \, + \, \frac{1}{2} \left(1+ \frac{1}{K \lambda^{-2}} \right) \, \left(  L^{3/2} h |x|^2 +  |v|^2 \right) \, L \, h^3  \;,
\end{equation}
for all $y = ((x,v), (\tilde{x}, \tilde{v})) \in \mathbb{R}^{4 d}$.  
\end{lemma}

The proof of Lemma~\ref{lem:inf_drif_mjp} is deferred to Section~\ref{sec:randomized_uHMC_proofs}.  Let $(p_t)_{t \ge 0}$ denote the transition semigroup of the exact process $((Q_t, p_t))_{t \ge 0}$.  We are now in position to quantify the asymptotic bias of randomized uHMC with sMC time integration.

\begin{theorem}
\label{thm:asympbias_mjp}
Suppose that Assumptions~\ref{A1}-\ref{A3} hold and $\lambda>0$ and $h>0$ satisfy \eqref{eq:Cl} and \eqref{eq:Ch}.  Then 
\begin{equation} \label{eq:asympbias_mjp}
\W^2(\mu, \tilde{\mu})^2 \, \le \, 8 \, \gamma^{-1} \left(1+ \frac{\lambda^2}{K} \right)  \, L \, \left(L^{3/2} K^{-1} h^4 d + h^3 d \right) \;.
\end{equation}
\end{theorem}

\begin{remark}[Why duration randomization?] \label{rmk:durationrandomization}
Since the number of jumps of the randomized uHMC process over $[0,t]$ is a Poisson process with intensity $\lambda + h^{-1}$, the mean number of steps  of $\Theta_h$ (and hence, gradient evaluations) over a time interval of length $t$ is $t/h$. Let $\nu$ be the initial distribution of the randomized uHMC process.  We choose $\lambda$ to saturate  \eqref{eq:Cl}, i.e., $\lambda = 12 L^{1/2}$.  The contraction rate in \eqref{eq:inf_contr_mjp} then becomes \[ \gamma \ = \   120^{-1} K^{1/2}  \left( \frac{K}{L} \right)^{1/2} \;.   \] 
According to Theorem~\ref{thm:contr_mjp}, to obtain $\varepsilon$-accuracy in $\W^2$ w.r.t.\ $\tilde{\mu}$, $t$ can be chosen such that  \begin{equation}
t \ = \ 240 K^{-1/2} \left( \frac{L}{K} \right)^{1/2}  \log\left( \frac{3 \max( 12^{-1} L^{1/2}, 12 L^{-1/2}) \W^2(\nu, \tilde{\mu}_{BG})}{\varepsilon} \right)^+ \;. \label{eq:comp_t_mjp} 
\end{equation} However, since $\tilde{\mu}$  is inexact, to resolve the asymptotic bias to $\varepsilon$-accuracy in $\W^2$, Theorem~\ref{thm:asympbias_mjp} indicates that it suffices to choose $h$ such that 
\begin{equation*} 8 \cdot 120 \cdot 13 \cdot \left( K d \left(\frac{L}{K} \right)^4 h^4 + K^{1/2} d \left(\frac{L}{K} \right)^{5/2} h^3  \right) \ \le \ \varepsilon^2 \;. 
\end{equation*}  In other words, it suffices to choose $h$ such that \begin{equation}
\label{eq:comp_h_mjp}
h^{-1} \ge 2 \max\left( (8 \cdot 195)^{1/4} K^{1/2} \left( \frac{d}{K} \right)^{1/4} \frac{L}{K} \varepsilon^{-1/2}, 2 (390)^{1/3} K^{1/2} \left( \frac{d}{K} \right)^{1/3} \left(\frac{L}{K} \right)^{5/6}  \varepsilon^{-2/3} \right) \;.
\end{equation}
Combining \eqref{eq:comp_t_mjp} and \eqref{eq:comp_h_mjp} gives an overall complexity  of \begin{equation} \label{eq:comp_mjp}
\begin{aligned}
\scriptsize
\frac{t}{h} \propto \max\left( \left( \frac{d}{K}  \right)^{1/4}  \left(\frac{L}{K} \right)^{3/2} \varepsilon^{-1/2} , \left( \frac{d}{K} \right)^{1/3} \left(\frac{L}{K} \right)^{4/3} \varepsilon^{-2/3}   \right)\times \log\left( \frac{\max(L^{1/2}, L^{-1/2}) \W^2(\nu, \tilde{\mu}_{BG})}{\varepsilon} \right)^+  \;.
\end{aligned}
\end{equation}
Note, this complexity guarantee has a better dependence on the condition number to the one obtained in Remark~\ref{rmk:complexity} for uHMC with fixed duration. Additionally, in the high condition number regime $L/K > (d/K)^{1/2} \epsilon^{-1}$,  the dependence on both the dimension and accuracy also improves.  
\end{remark}

\begin{proof}[Proof of Theorem~\ref{thm:asympbias_mjp}]
Let $(Y_t)_{t \ge 0}$ be the coupling process generated by $\mathcal{A}^C$ in \eqref{eq:AC_asympbias} with initial condition $Y_0 = ((Q_0, V_0), (\tilde{Q}_0, \tilde{V}_0)) \sim \mu_{BG} \otimes \tilde{\mu}_{BG}$.  Fix $t>0$.  Then by the coupling characterization of the $L^2$-Wasserstein distance, and It\^{o}'s formula for jump processes  applied to $t \mapsto e^{\gamma t /2} \rho(Y_t)^2$, we obtain
   \begin{align*}
&\W^2(\mu, \tilde{\mu})^2 \le E( |Q_0 - \tilde{Q}_0|^2 )  \le  \frac{16}{3} E\left( \frac{3}{16} |Q_0 - \tilde{Q}_0|^2 + \lambda^{-2} \left| V_0 - \tilde{V}_0 + \frac{\lambda}{4} (Q_0 - \tilde{Q}_0) \right|^2 \right)  \\
& \, \le \,  \frac{16}{3} E \rho(Y_0)^2   = \frac{16}{3} E \rho(Y_t)^2 \\
& \, \le \, 
\frac{16}{3} 
\left( e^{-\gamma t/2} E \rho(Y_0)^2 + \int_0^t e^{-\gamma (t-s)/2} \left( \frac{\gamma}{2} \rho(Y_s)^2 + \mathcal{A}^C \rho(Y_s)^2 \right) ds \right)  \\
& \, \overset{Lem.~\ref{lem:inf_drif_mjp}}{\le} \, 
8 e^{-\gamma t/2} E \rho(Y_0)^2  + 4 \, \left(1+ \frac{1}{K \lambda^{-2}} \right)  \, L \, h^3 \, \int_0^t e^{-\gamma (t-s)/2} \left( L^{3/2} h E|Q_s|^2 + E|p_s|^2 \right) ds 
\;.
\end{align*} 
Since the exact process leaves $\mu_{BG}$ invariant, the integrand in this expression simplifies
\begin{align*}
& \W^2(\mu, \tilde{\mu})^2
 \, \le \, 
8 e^{-\gamma t/2} E \rho(Y_0)^2  + 4 \, \left(1+ \frac{1}{K \lambda^{-2}} \right)   \, L \, h^3 \,  \left(   L^{3/2} h E|Q_0|^2 + d \right) \,  \int_0^t e^{-\gamma (t-s)/2}  ds  \\
 & \qquad \overset{Rem.~\ref{rmk:remark1}}{\le} \, 8 \, \gamma^{-1} \left(1+ \frac{\lambda^2}{K} \right)  \, \frac{L}{K} \, h^3  \, d \, \left(  L^{3/2} h + K \right) \;. \end{align*} 
Simplifying this expression  gives \eqref{eq:asympbias_mjp}.
\end{proof}

\subsection{Proofs for Randomized uHMC}

\label{sec:randomized_uHMC_proofs}

\begin{proof}[Proof of Lemma~\ref{lem:inf_contr_mjp}]

Let $F_{\mathcal{U}} := F(x+ \mathcal{U} v)$, $\tilde{F}_{\mathcal{U}} := F(\tilde{x}+ \mathcal{U} \tilde{v})$,  $Z_{\mathcal{U}} := z + \mathcal{U} w$, and  $\Delta F_{\mathcal{U}} := F_{\mathcal{U}} - \tilde{F}_{\mathcal{U}}$.
Note that by \ref{A2}, \ref{A3} and \eqref{eq:cocoercive}, \begin{equation} \label{ieq:coco}    K \, |Z_{\mathcal{U}}|^2 \, \overset{\ref{A3}}{\le} \, - \langle Z_{\mathcal{U}}, \Delta F_{\mathcal{U}}\rangle   \,  \overset{\ref{A2}}{\le} \, L \, |Z_{\mathcal{U}}|^2 \, \;, \quad  |\Delta F_{\mathcal{U}}|^2 \, \overset{\eqref{eq:cocoercive}}{\le} \, - L  \, \langle Z_{\mathcal{U}}, \Delta F_{\mathcal{U}} \rangle   \;. \end{equation} 
The idea of this proof is to decompose $\tilde{\mathcal{G}}^C \rho(y)^2$ into two terms: a gain $\Gamma_0$ and loss $\Lambda_0$, and use \eqref{ieq:coco} and the hyperparameter assumptions, to obtain an overall gain.

\smallskip

To this end, evaluate \eqref{eq:GC_contr} at $f(y) = \rho(y)^2$ to obtain, \begin{align}
& \tilde{\mathcal{G}}^C \rho(y)^2 \, = \,  h^{-1} E \left(  \rho(\Theta_h(x,v,\mathcal{U}),\Theta_h(\tilde{x},\tilde{v},\mathcal{U}))^2 - \rho(y)^2 \right)  \nonumber \\
& \quad +  \lambda E \left( \rho( (x,\xi), (\tilde{x}, \xi) )^2 - \rho( y )^2 \right) \nonumber \\
 & \, = \, \ \
 \lambda^{-1} \left( \frac{\lambda h}{4} + \frac{1}{2} \right) E \langle Z_{\mathcal{U}}, \Delta F_{\mathcal{U}} \rangle \nonumber \\
 & \quad + \lambda^{-1} \left( -\frac{1}{2} + \frac{\lambda h}{4} \right) |w|^2  +  \lambda^{-1} E \left[ \left( \frac{\lambda^2 h^2}{4} +  \frac{3 \lambda h}{4} + 2 - \lambda \mathcal{U} ( \frac{\lambda h}{4}  + \frac{1}{2} )  \right) \lambda^{-1}  \langle w, \Delta F_{\mathcal{U}} \rangle \right]  \nonumber  \\
 & \quad  + \lambda^{-1}  \left( \frac{\lambda^3 h^3}{16} + \frac{\lambda^2 h^2}{4} +  \lambda h \right) \lambda^{-2} E |\Delta F_{\mathcal{U}}|^2  \, = \, \Gamma_0 + \Lambda_0 \quad \text{where} \nonumber \\
& \Gamma_0 :=  \lambda^{-1} \left( \frac{\lambda h}{4} + \frac{1}{2} \right) E \langle Z_{\mathcal{U}}, \Delta F_{\mathcal{U}} \rangle -  \lambda^{-1} \frac{1}{2} |w|^2 \label{eq:gain_1} \\
& \Lambda_0 := \lambda^{-1} \left( \frac{\lambda^3 h^3}{16} + \frac{\lambda^2 h^2}{4} +  \lambda h   \right) \lambda^{-2} E |\Delta F_{\mathcal{U}}|^2 
   +  \lambda^{-1} \left(  \frac{\lambda h}{4}   \right) |w|^2 \label{eq:loss_1}  \\
    & \quad +  \lambda^{-1} E \left[ \left( \frac{\lambda^2 h^2}{4} +  \frac{3 \lambda h}{4} + 2 - \lambda \mathcal{U} ( \frac{\lambda h}{4}  + \frac{1}{2} )  \right) \lambda^{-1}  \langle w, \Delta F_{\mathcal{U}} \rangle \right] \;.  \nonumber
\end{align}
Note that the term $\lambda^{-1} E \left[\lambda \mathcal{U} ( \frac{\lambda h}{4}  + \frac{1}{2} )  \lambda^{-1}  \langle w, \Delta F_{\mathcal{U}} \rangle \right]$ was added and subtracted in order to leverage the co-coercivity property of $\nabla U$ at $Z_{\mathcal{U}}$; see \eqref{ieq:coco}.
By the Peter-Paul inequality with parameter $(L \lambda^{-2})^{1/2}$,  $\lambda h \le 1$, and \eqref{ieq:coco}, \begin{align}
\Lambda_0 &  \,   \le \, \lambda^{-1} \left( \frac{\lambda^3 h^3}{16} + \frac{\lambda^2 h^2}{4} +  \lambda h  + \frac{1}{2 (L \lambda^{-2})^{1/2}} \left( \frac{\lambda^2 h^2 }{4}  + \frac{3 \lambda h}{4} + 2  \right) \right) \lambda^{-2} E |\Delta F_{\mathcal{U}}|^2 \nonumber
 \\
 & \quad   +  \lambda^{-1} \left(  \frac{\lambda h}{4} + \frac{(L \lambda^{-2})^{1/2}}{2} \left( \frac{\lambda^2 h^2}{4}  + \frac{3 \lambda h}{4} + 2 \right)  \right) |w|^2 \nonumber \\
& \, \overset{\eqref{eq:Ch}}{\le} \,  \lambda^{-1} \left( \frac{21}{16} + \frac{3}{2 (L \lambda^{-2})^{1/2}}  \right) \lambda^{-2} E |\Delta F_{\mathcal{U}}|^2 + \lambda^{-1} \left( \frac{1}{4}  + \frac{3 (L \lambda^{-2})^{1/2}}{2}  \right) |w|^2  \nonumber  \\
 \, & \overset{\eqref{ieq:coco}}{\le} \, - \lambda^{-1} \left(   \frac{21 L \lambda^{-2}}{16} + \frac{3 (L \lambda^{-2})^{1/2}}{2}  \right)  E \langle Z_{\mathcal{U}}, \Delta F_{\mathcal{U}} \rangle  + \lambda^{-1} \left( \frac{1}{4}  + \frac{3 (L \lambda^{-2})^{1/2}}{2}  \right) |w|^2 \nonumber \\
   \, & \overset{\eqref{eq:Cl}}{\le} \, - \lambda^{-1} \left(   \frac{21 }{16 \cdot 12^2} + \frac{1}{8}  \right)  E \langle Z_{\mathcal{U}}, \Delta F_{\mathcal{U}} \rangle  +   \lambda^{-1} \frac{3 }{8}  |w|^2  \;. \label{ieq:loss_1}
\end{align}
Combining $\Gamma_0$ in \eqref{eq:gain_1} with the upper bound on $\Lambda_0$ in \eqref{ieq:loss_1} yields
an overall gain
\begin{align*}
\tilde{\mathcal{G}}^C \rho(y)^2  \, \le \, - \frac{1}{8} \lambda^{-1} \left( -\frac{5}{2} E \langle Z_{\mathcal{U}}, \Delta F_{\mathcal{U}} \rangle + |w|^2 \right)  \, & \overset{\eqref{ieq:coco}}{\le} \, - \frac{1}{8} \lambda^{-1} \left( \frac{5}{2} K E |Z_{\mathcal{U}}|^2 + |w|^2 \right) \;.
\end{align*}
 By the Peter-Paul inequality with parameter $h$ and noting that ${\mathcal{U} \sim \operatorname{Uniform}(0,h)}$, \begin{align*} 
 \tilde{\mathcal{G}}^C \rho(y)^2 \, &\le \, - \frac{1}{8} \lambda^{-1} \left( \frac{5}{2}  K ( |z|^2 + h \langle z, w \rangle + \frac{h^2}{3} |w|^2)  +  |w|^2 \right) \\
 \, & \le \, - \frac{1}{8} \lambda^{-1} \left(  \frac{5}{4} K  |z|^2   + \left( 1 - \frac{5}{12} K h^2 \right) |w|^2 \right) \overset{\eqref{eq:Cl}}{\le} - \frac{1}{9} \lambda^{-1} \left( K |z|^2 + |w|^2 \right) \;, 
\end{align*} where in the last step we used  $K \le L$, $\lambda h \le 1$ and  $(L \lambda^{-2} )^{1/2} \le 12^{-1}$.  Inserting \eqref{eq:straight_metric} and simplifying yields the required infinitesimal contraction result, i.e., \begin{align*} 
 \tilde{\mathcal{G}}^C \rho(y)^2 \le   -\frac{1}{9} \lambda^{-1} \min(\frac{12}{13}  \lambda^2, K  )  \rho(y)^2  \le  -\frac{1}{10}  \min(  \lambda, K \lambda^{-1} )  \rho(y)^2 
 =  -\frac{K}{10 \lambda}   \rho(y)^2
 \;,
\end{align*} 
since $K \lambda^{-2} \le 12^{-1} < 1$ implies $K < \lambda^2$, and hence, $\min(\lambda, K \lambda^{-1}) = K \lambda^{-1}$.  
\end{proof}

%The following Lemma is used to prove Lemma~\ref{lem:inf_drif_mjp}. It quantifies the single-step $L^2$-accuracy of the sMC integrator w.r.t.\ the distorted metric $\rho(y)^2$ from \eqref{eq:dist_metric}.

%\begin{lemma} \label{lem:1step_L2accuracy_sMC} Suppose that Assumptions~\ref{A1}-\ref{A2} hold and $\lambda>0$ and $h>0$ satisfy \eqref{eq:Cl} and \eqref{eq:Ch}, respectively.  Then for all $(x,v) \in \mathbb{R}^{2d}$, \begin{equation} \label{eq:1step_L2accuracy_sMC} E \rho(\theta_h(x,v), \Theta_h(x,v,\mathcal{U}) )^2  \, \le  \, 12^{-1}   \, L \, (  |v|^2  + L |x|^2  ) \, h^{4}  \;. \end{equation} \end{lemma}

\begin{proof}[Proof of Lemma~\ref{lem:inf_drif_mjp}]
Let $F_{\mathcal{U}} := F(x+ \mathcal{U} v)$, $\tilde{F}_{\mathcal{U}} := F(\tilde{x}+ \mathcal{U} \tilde{v})$, $Z_{\mathcal{U}} := z + \mathcal{U} w$, and  $\Delta F_{\mathcal{U}} := F_{\mathcal{U}} - \tilde{F}_{\mathcal{U}}$.  The idea of this proof is related to the proof of Lemma~\ref{lem:inf_contr_mjp}: we carefully decompose $\mathcal{A}^C \rho(y)^2$ into a gain $\Gamma_0$, loss $\Lambda$, and also, a discretization error $\Delta$, and use \eqref{ieq:coco} and the hyperparameter assumptions, to obtain a gain from the contractivity of the underlying randomized uHMC process up to discretization error. This estimate results in an infinitesimal drift condition, as opposted to an infinitesimal contractivity result. The quantification of the discretization error is related to the $L^2$-error estimates for the sMC time integrator developed in Lemma~\ref{lem:smc_strong_accuracy}, though the semi-exact flow only implicitly appears below, since the proof involves a one step analysis.      

\smallskip

As a preliminary step, we develop some estimates that are used to bound the discretization error. Recall that $(q_s(x,v),p_s(x,v))$ denotes the exact Hamiltonian flow.   Since $L h^2 \le 12^{-2} \le 1/4$ (by the hypotheses: $(L \lambda^{-2})^{1/2} \le 1/12$ and $\lambda h \le 1$),  \eqref{apriori:2} and the Cauchy-Schwarz inequality imply that \begin{align}
\label{apriori:2aa}
\sup_{s \le h} |p_s(x,v)|^2 &\le 3 |v|^2 + 4 L^2 h^2 (|x|^2 + h^2 |v|^2) \le 4 L^2 h^2  |x|^2   + \frac{31}{10} |v|^2  \;.
\end{align}
As a shorthand, define   \begin{align*}
F_1 &:= \frac{2}{h^2} \int_0^h (h-s) F(q_s(x,v)) ds \;,  ~  \Delta F_1 := \frac{2}{h^2} \int_0^h (h-s) [F(q_s(x,v)) - F_{\mathcal{U}}] ds \;,   \\
F_2 &:= \frac{1}{h} \int_0^h  F(q_s(x,v)) ds \;, ~ \text{and} ~~ \Delta F_2 := \frac{1}{h} \int_0^h [F(q_s(x,v)) - F_{\mathcal{U}}] ds \;. \end{align*}
Then by the Cauchy-Schwarz  inequality \begin{align}
E|\Delta F_1|^2 \, &= \, E \left| \frac{2}{h^2} \int_0^h (h-s) [ F(q_s) - F_{\mathcal{U}}  ] ds \right|^2  
\, \le \, \frac{4}{h^4} \frac{h^3}{3} \int_0^h E | F(q_s) - F_{\mathcal{U}}  |^2 ds 
\nonumber \\
\, & \overset{\ref{A2}}{\le} \, \frac{4}{h^4} \frac{L^2 h^3}{3} \int_0^h E | q_s- x - \mathcal{U} v  |^2 ds \, = \,  \frac{4}{h^4} \frac{L^2 h^3}{3} \int_0^h E | \int_0^s v_r dr - \mathcal{U} v  |^2 ds \nonumber \\
\, & \le \,  \frac{4}{h^4} \frac{2 L^2 h^3}{3}  \left(  \int_0^h |\int_0^s v_r dr|^2 ds +  h^3 |v|^2 \right) \, \le \,   \frac{8 L^2 h^2}{3}  \left(  \sup_{s \le h} |p_s|^2  +  |v|^2 \right) \nonumber \\
\, & \overset{\eqref{apriori:2aa}}{\le} \,  \frac{8 L^2 h^2}{3} \left(   4 L^2 h^2 |x|^2  +  \frac{41}{10} |v|^2 \right)  
\label{eq:delta_F_1}
\end{align}
where in the next to last step we used Young's product inequality.  Similarly, \begin{align}
E|\Delta F_2|^2 \, &= \, E \left| \frac{1}{h} \int_0^h [F(q_s(x,v))  - F_{\mathcal{U}}] ds
\right|^2 \, \le  \, h^{-1} E \int_0^h |F(q_s(x,v))  - F_{\mathcal{U}}|^2 ds \nonumber \\
\, &\overset{\ref{A2}}{\le} \, \frac{L^2}{h}  \int_0^h E | \int_0^s v_r dr - \mathcal{U} v  |^2 ds \, \le \, \frac{2 L^2}{h}  \left( \int_0^h | \int_0^s v_r dr|^2 ds +  h^3 | v  |^2 \right) \nonumber \\
\, &\le  2 L^2 h^2 \left(   \sup_{s \le h} |p_s|^2  +  |v|^2   \right)  
\overset{\eqref{apriori:2aa}}{\le}
2 L^2 h^2  \left( 4 L^2 h^2 |x|^2 + \frac{41}{10} |v|^2 \right) 
\;.  \label{eq:delta_F_2}
\end{align}
Combining \eqref{eq:delta_F_1} and \eqref{eq:delta_F_2} we obtain \begin{equation}
E|\Delta F_1|^2  \vee E|\Delta F_2|^2  \le \frac{8}{3} L^2 h^2  \left( 4 L^2 h^2 |x|^2 + \frac{41}{10} |v|^2 \right)  \le 11 L^2 h^2  \left( L^2 h^2 |x|^2 +  |v|^2 \right)  \;.  \label{eq:delta_F}
\end{equation}
In order to obtain a sharp error estimate for the sMC time integrator, the following upper bound is crucial \begin{align}
| E \Delta F_2 | \, &= \,  \left| \frac{1}{h} E \int_0^h [F(q_s(x,v))  - F_{\mathcal{U}}] ds
\right|  = \left| \frac{1}{h} \int_0^h [F(q_s(x,v))  - F(x+s v)] ds
\right| \nonumber \\
\, & \overset{\ref{A2}}{\le} \,  \frac{L}{h} \int_0^h |q_s(x,v)  - (x+s v)| ds \, \le \,  \frac{L}{h} \int_0^h \left| \int_0^s (s-r) F(q_r(x,v)) dr \right| ds \nonumber \\
\, & \overset{\ref{A2}}{\le} \, 
\frac{L}{h} \int_0^h  \int_0^s (s-r) |F(q_r(x,v))| dr  ds \, \le \, \frac{L^2 h^2}{6} ( |x| + h \sup_{s \le h} |p_s| )  \nonumber \\
\,  & \overset{\eqref{apriori:2}}{\le} \, \frac{L^2 h^2}{6} (|x| + h [ |v| + L h (1+L h^2) (|x|+ h |v| )])
\, \overset{\eqref{eq:Cl}}{\le} \, \frac{L^2 h^2}{5} (   |x| + h |v| ) \;.  \nonumber
\end{align}
Thus, by Cauchy-Schwarz inequality, \begin{equation}
| E \Delta F_2 |^2 \, \le \, 
\frac{2 L^4 h^4}{25} (  |x|^2 + h^2 |v|^2 ) \, \overset{\eqref{eq:Cl}}{\le} \, 
\frac{L^{3} h^4}{12} (  L |x|^2 +  |v|^2 ) \;,
\label{eq:mean_delta_F_2}
\end{equation}
in the last step the numerical pre-factor was simplified by using $\lambda h \le 1$ and $(L \lambda^{-2})^{1/2} \le 12^{-1}$.  

\smallskip

Evaluate \eqref{eq:AC_asympbias} at $f(y) = \rho(y)^2$ and expand to obtain \begin{align}
& \mathcal{A}^C \rho(y)^2 \, = \,  h^{-1} E \left(  \rho(\theta_h(x,v),\Theta_h(\tilde{x},\tilde{v},\mathcal{U}))^2 - \rho(y)^2 \right) \nonumber  \\
 & \qquad \qquad +
 \lambda E \left( \rho( (x,\xi), (\tilde{x}, \xi) )^2 - \rho( y )^2 \right) \nonumber \\
 & \, = \, \ \
 \lambda^{-1}  \frac{1}{2}  E \langle z, F_2-\tilde{F}_{\mathcal{U}} \rangle + 
  \lambda^{-1} \frac{\lambda h}{4} E \langle z, F_1-\tilde{F}_{\mathcal{U}} \rangle  + \lambda^{-1} \left( -\frac{1}{2} + \frac{\lambda h}{4} \right) |w|^2
\nonumber \\
 & \quad +  \lambda^{-1} E  \left( \frac{\lambda^2 h^2}{4} +  \frac{\lambda h}{4}   \right) \lambda^{-1}  E\langle w, F_1-\tilde{F}_{\mathcal{U}} \rangle    +  \lambda^{-1} \left( \frac{\lambda h}{2} + 2   \right) \lambda^{-1}  E\langle w, F_2-\tilde{F}_{\mathcal{U}} \rangle   \nonumber  \\
   & \quad  + \lambda^{-1}  \left( \frac{\lambda^3 h^3}{16}  \right) \lambda^{-2} E |F_1-\tilde{F}_{\mathcal{U}}|^2    + \lambda^{-1} \,  (\lambda h) \, \lambda^{-2} E |F_2-\tilde{F}_{\mathcal{U}}|^2  \nonumber \\
   & \quad  + \lambda^{-1}  \left( \frac{\lambda^2 h^2}{4}  \right) \lambda^{-2} E \langle F_1-\tilde{F}_{\mathcal{U}}, F_2-\tilde{F}_{\mathcal{U}} \rangle  \nonumber \\
\, &= \, \Gamma_0 + \Lambda_0 
 + \lambda^{-1} \frac{\lambda h}{4} E \langle z, \Delta F_1 \rangle  + \lambda^{-1}  \frac{1}{2}  \langle z,   E\Delta F_2 \rangle 
\nonumber \\
 & \quad +  \lambda^{-1}  \left( \frac{\lambda^2 h^2}{4} +  \frac{\lambda h}{4}   \right) E \langle w,\lambda^{-1}  \Delta F_1 \rangle    +  \lambda^{-1}  \left( \frac{\lambda h}{2} + 2   \right)   \langle w,\lambda^{-1} E\Delta F_2 \rangle 
 \nonumber  \\
   & \quad  + \lambda^{-1}  \left( \frac{\lambda^3 h^3}{8} + \frac{\lambda^2 h^2}{4}  \right)  E \langle \lambda^{-1} \Delta F_1, \lambda^{-1} \Delta F_{\mathcal{U}} \rangle   \nonumber \\
   & \quad + \lambda^{-1}  \left( \frac{\lambda^2 h^2}{4} + 2 \lambda h \right)  E \langle \lambda^{-1} \Delta F_2, \lambda^{-1} \Delta F_{\mathcal{U}} \rangle \nonumber \\
   & \quad  + \lambda^{-1}  \left( \frac{\lambda^3 h^3}{16}  \right) \lambda^{-2} E |\Delta F_1|^2    + \lambda^{-1} \,  (\lambda h) \, \lambda^{-2} E |\Delta F_2|^2  \nonumber \\
      & \quad  + \lambda^{-1}  \left( \frac{\lambda^2 h^2}{4}  \right)  E \langle \lambda^{-1} \Delta F_1, \lambda^{-1} \Delta F_2 \rangle \;,    \nonumber
\end{align}
where $\Gamma_0$ and $\Lambda_0$  are the loss and gain from $\tilde{\mathcal{G}}^C \rho(y)^2$ in \eqref{eq:gain_1} and \eqref{eq:loss_1}, respectively. 
Now we decompose $\mathcal{A}^C \rho(y)^2 $ 
\begin{align}
& \mathcal{A}^C \rho(y)^2   \, 
 \, \le \, \Gamma_0 + \Lambda + \Delta \quad \text{where}  \label{eq:gain_2} \\
& \Lambda \, := \,  \Lambda_0 + \lambda^{-1} \frac{K}{18}  |z|^2 + \lambda^{-1}  \frac{1}{18}  |w|^2  +  \lambda^{-1} \left( \frac{\lambda^3 h^3}{16}  + \frac{\lambda^2 h^2}{4} +  \lambda h  \right) \lambda^{-2} E|\Delta F_{\mathcal{U}}|^2  \;,  \label{eq:loss_2} \\
& \Delta \, := \,  \lambda^{-1} \left( \frac{9}{4 K \lambda^{-2}} + 36 \left( 1 + \frac{\lambda h}{4} \right)^2 \right) \lambda^{-2} |E \Delta F_2|^2 \nonumber \\
& \quad \, + \, \lambda^{-1} \left( \frac{\lambda^3 h^3}{8}  + \frac{\lambda^2 h^2}{4} 
+ \frac{9 \lambda^2 h^2}{16 K \lambda^{-2}} 
+ 36 \left( \frac{\lambda^2 h^2}{8} + \frac{\lambda h}{8} \right)^2  \right) \lambda^{-2} E |\Delta F_1|^2 \nonumber \\ 
& \quad  \, + \, \lambda^{-1} \left( 
2 \lambda h  + \frac{\lambda^2 h^2}{4}  
\right) \lambda^{-2} E |\Delta F_2 |^2  \;.\label{eq:error_2} 
\end{align}
Here the loss $\Lambda$ in $\mathcal{A}^C \rho(y)^2$ was separated from the discretization error $\Delta$ by applying the Peter-Paul inequality to the four cross terms involving $z$ or $w$ (i.e., $E \langle z, \Delta F_1 \rangle$, $E \langle z,  \Delta F_2 \rangle$, $E \langle w,\lambda^{-1}  \Delta F_1 \rangle $, and $E \langle w,\lambda^{-1} \Delta F_2 \rangle $), and Young's product inequality for the remaining cross terms. As expected, the gain in \eqref{eq:gain_2} is the same as the gain in \eqref{eq:gain_1}.  We next bound the loss and discretization error terms separately.

\smallskip

For the loss $\Lambda$ in \eqref{eq:loss_2}, apply the Peter-Paul inequality with parameter $(L \lambda^{-2})^{1/2}$,  $\lambda h \le 1$, and \eqref{ieq:coco}, to obtain
\begin{align}
& \Lambda  \, \le \,  \lambda^{-1} \frac{K}{18}  |z|^2 + \lambda^{-1}  \frac{1}{18}  |w|^2 \nonumber \\
& \quad + \lambda^{-1} \left( \frac{\lambda^3 h^3}{8} + \frac{\lambda^2 h^2}{2} +  2 \lambda h  + \frac{1}{2 (L \lambda^{-2})^{1/2}} \left( \frac{\lambda^2 h^2 }{4}  + \frac{3 \lambda h}{4} + 2  \right) \right) \lambda^{-2} E |\Delta F_{\mathcal{U}}|^2 \nonumber
 \\
 & \quad   +  \lambda^{-1} \left(  \frac{\lambda h}{4} + \frac{(L \lambda^{-2})^{1/2}}{2} \left( \frac{\lambda^2 h^2}{4}  + \frac{3 \lambda h}{4} + 2 \right)  \right) |w|^2 \nonumber \\
& \, \overset{\eqref{eq:Ch}}{\le} \,  \lambda^{-1} \frac{K}{18}  |z|^2 + \lambda^{-1}  \frac{1}{18}  |w|^2 \nonumber \\
& \quad + \lambda^{-1} \left( \frac{21}{8} + \frac{3}{2 (L \lambda^{-2})^{1/2}}  \right) \lambda^{-2} E |\Delta F_{\mathcal{U}}|^2 + \lambda^{-1} \left( \frac{1}{4}  + \frac{3 (L \lambda^{-2})^{1/2}}{2}  \right) |w|^2  \nonumber  \\
 \, & \overset{\eqref{ieq:coco}}{\le} \, \lambda^{-1} \frac{K}{18}  |z|^2 + \lambda^{-1}  \frac{1}{18}  |w|^2 \nonumber \\
& \quad - \lambda^{-1} \left(   \frac{21 L \lambda^{-2}}{8} + \frac{3 (L \lambda^{-2})^{1/2}}{2}  \right)  E \langle Z_{\mathcal{U}}, \Delta F_{\mathcal{U}} \rangle  + \lambda^{-1} \left( \frac{1}{4}  + \frac{3 (L \lambda^{-2})^{1/2}}{2}  \right) |w|^2 \nonumber \\
   \, & \overset{\eqref{eq:Cl}}{\le} \,    - \lambda^{-1} \left(   \frac{21 }{8 \cdot 12^2} + \frac{1}{8}  \right)  E \langle Z_{\mathcal{U}}, \Delta F_{\mathcal{U}} \rangle + \lambda^{-1} \frac{K}{18}  |z|^2   +   \lambda^{-1} \frac{31}{72}  |w|^2  \;. 
\label{ieq:loss_2}
\end{align}
Combining \eqref{eq:gain_2} with \eqref{ieq:loss_2}, and following the corresponding steps in the proof of Lemma~\ref{lem:inf_contr_mjp}, we obtain \begin{align}
\Gamma_0 + \Lambda  & \,  \le \,     \lambda^{-1} \left( \frac{137}{384}  \right) E \langle Z_{\mathcal{U}}, \Delta F_{\mathcal{U}} \rangle 
+ \lambda^{-1} \frac{K}{18} |z|^2
- \lambda^{-1} \frac{5}{72} |w|^2 \nonumber \\ 
& \, \overset{\eqref{ieq:coco}}{\le} \,    -\lambda^{-1} \left( \frac{137}{384}  \right) K E | Z_{\mathcal{U}} |^2 
+ \lambda^{-1} \frac{K}{18} |z|^2
- \lambda^{-1} \frac{5}{72} |w|^2 \nonumber \\
& \, \le \,     -\lambda^{-1} \left( \frac{137}{384}  \right) K  ( |z|^2 + h \langle z, w \rangle + \frac{h^2}{3} |w|^2)
+ \lambda^{-1} \frac{K}{18} |z|^2
- \lambda^{-1} \frac{5}{72} |w|^2 \nonumber \\
& \, \le \,     -\lambda^{-1} \left( \frac{137}{384}  \right) K  \left( \frac{1}{2} |z|^2 +  \left( \frac{h^2}{3} - \frac{ h^2}{2} \right) |w|^2 \right)
+ \lambda^{-1} \frac{K}{18} |z|^2
- \lambda^{-1} \frac{5}{72} |w|^2 \nonumber \\
& \, \overset{\eqref{eq:Cl}}{\le} \,     - \frac{1}{16} \lambda^{-1} \left( K |z|^2 + |w|^2 \right) \, \overset{\eqref{eq:straight_metric}}{\le}   - \frac{1}{16} \lambda^{-1} \min(\frac{12}{13}  \lambda^2, K  ) \, \rho(y)^2 \nonumber \\
& \, \le \, -\frac{1}{18}  \min(  \lambda, K \lambda^{-1} ) \, \rho(y)^2 \;.
\label{ieq:gain_loss_2}
\end{align}
Note that the numerical pre-factors appearing in the last few steps were simplified for readability.  

\smallskip

For the discretization error $\Delta$  in \eqref{eq:error_2}, 
apply $\lambda h \le 1$, and
insert \eqref{eq:delta_F} and \eqref{eq:mean_delta_F_2}, \begin{align}
\Delta  & \, = \,  \left( \frac{\lambda^2 h^2}{8}  + \frac{\lambda h}{4} + 36 \lambda h \left( \frac{\lambda h}{8} + \frac{1}{8} \right)^2
+ \frac{9 \lambda h}{16 K \lambda^{-2}} 
   \right) \lambda^{-2} h E |\Delta F_1|^2   \nonumber \\
& \quad \, + \, 
\left( 
2   + \frac{\lambda h}{4}  
\right) \lambda^{-2} h E |\Delta F_2 |^2  
 \, + \, \lambda^{-1} \left( 36 \left( 1 + \frac{\lambda h}{4} \right)^2 + \frac{9}{4 K \lambda^{-2}}  \right) \lambda^{-2} |E \Delta F_2|^2  \nonumber \\
& \, \overset{\eqref{eq:Ch}}{\le} \,   
 \left( \frac{21}{8}
+ \frac{9}{16 K \lambda^{-2}} 
  \right) \lambda^{-2} h E |\Delta F_1|^2  +  \left( 
 \frac{9}{4}  
\right) \lambda^{-2} h E |\Delta F_2 |^2 \nonumber \\
& \quad \, + \, 
\lambda^{-1} \left( 
\frac{225}{4}
+ \frac{9}{4 K \lambda^{-2}} 
\right) \lambda^{-2}  |E\Delta F_2 |^2
\nonumber \\
& \, \le \, 
11   \left( \frac{39}{8} + \frac{9}{16 K \lambda^{-2}}  
\right) \, (L \lambda^{-2})   \, L \, h^3  \, 
\left( L^2 h^2 |x|^2 +  |v|^2 \right) \nonumber \\
& \quad \, + \, 
\frac{1}{12} \left( 
\frac{225}{4}
+ \frac{9}{4 K \lambda^{-2}} 
\right)   ( L \lambda^{-2} )^{3/2}  
 \, L^{3/2} \, h^4  \, 
\left( L |x|^2 +  |v|^2 \right) 
\nonumber \\
& \, \overset{\eqref{eq:Cl}}{\le} \, \frac{1}{2} \left( 1 + \frac{1}{K \lambda^{-2}} \right)
\, L \, h^3  \, 
\left( L^{3/2} h |x|^2 +  |v|^2 \right)
\label{ieq:error_2}
\end{align}
where in the last step several of the numerical pre-factors were rounded up to unity for readability.

\smallskip

Combining \eqref{ieq:gain_loss_2} and \eqref{ieq:error_2} gives \eqref{eq:inf_drif_mjp} --- as required.
\end{proof}

%\printbibliography

\bibliographystyle{imsart-number} % Style BST file (imsart-number.bst or imsart-nameyear.bst)
\bibliography{nawaf.bib}

\begin{thebibliography}{43}
% BibTex style file: imsart-number.bst, 2017-11-03
% Default style options (sort=1,type=number).
% Used options (sort=1,type=number).

\bibitem{achleitner2015large}
\begin{barticle}[author]
\bauthor{\bsnm{Achleitner},~\bfnm{Franz}\binits{F.}},
  \bauthor{\bsnm{Arnold},~\bfnm{Anton}\binits{A.}} \AND
  \bauthor{\bsnm{St{\"u}rzer},~\bfnm{Dominik}\binits{D.}}
(\byear{2015}).
\btitle{Large-time behavior in non-symmetric Fokker-Planck equations}.
\bjournal{arXiv preprint arXiv:1506.02470}.
\end{barticle}
\endbibitem

\bibitem{BlCaSa2014}
\begin{barticle}[author]
\bauthor{\bsnm{Blanes},~\bfnm{S.}\binits{S.}},
  \bauthor{\bsnm{Casas},~\bfnm{F.}\binits{F.}} \AND
  \bauthor{\bsnm{Sanz-Serna},~\bfnm{J.~M.}\binits{J.~M.}}
(\byear{2014}).
\btitle{Numerical integrators for the Hybrid {M}onte {C}arlo method}.
\bjournal{SIAM Journal on Scientific Computing}
\bvolume{36}
\bpages{A1556--A1580}.
\end{barticle}
\endbibitem

\bibitem{BouRabeeEberleLectureNotes2020}
\begin{bmisc}[author]
\bauthor{\bsnm{Bou-Rabee},~\bfnm{N.}\binits{N.}} \AND
  \bauthor{\bsnm{Eberle},~\bfnm{A.}\binits{A.}}
(\byear{2020}).
\btitle{Markov Chain Monte Carlo Methods}.
\bnote{URL: \url{https://wt.iam.uni-bonn.de/eberle/}.}
\end{bmisc}
\endbibitem

\bibitem{BoEb2022}
\begin{barticle}[author]
\bauthor{\bsnm{Bou-Rabee},~\bfnm{N.}\binits{N.}} \AND
  \bauthor{\bsnm{Eberle},~\bfnm{A.}\binits{A.}}
(\byear{2022}).
\btitle{{Couplings for Andersen Dynamics in High Dimension}}.
\bjournal{Ann. Inst. H. Poincar{\'e} Probab. Statist}
\bvolume{58}
\bpages{916-944}.
\end{barticle}
\endbibitem

\bibitem{BoEbZi2020}
\begin{barticle}[author]
\bauthor{\bsnm{Bou-Rabee},~\bfnm{Nawaf}\binits{N.}},
  \bauthor{\bsnm{Eberle},~\bfnm{Andreas}\binits{A.}} \AND
  \bauthor{\bsnm{Zimmer},~\bfnm{Raphael}\binits{R.}}
(\byear{2020}).
\btitle{Coupling and convergence for Hamiltonian Monte Carlo}.
\bjournal{Ann. Appl. Probab.}
\bvolume{30}
\bpages{1209--1250}.
\end{barticle}
\endbibitem

\bibitem{BoSa2017}
\begin{barticle}[author]
\bauthor{\bsnm{Bou-Rabee},~\bfnm{Nawaf}\binits{N.}} \AND
  \bauthor{\bsnm{Sanz-Serna},~\bfnm{Jes{\'u}s~Mar{\'\i}a}\binits{J.~M.}}
(\byear{2017}).
\btitle{Randomized Hamiltonian Monte Carlo}.
\bjournal{Ann. Appl. Probab.}
\bvolume{27}
\bpages{2159--2194}.
\bdoi{10.1214/16-AAP1255}
\end{barticle}
\endbibitem

\bibitem{BoSaActaN2018}
\begin{barticle}[author]
\bauthor{\bsnm{B{ou-Rabee}},~\bfnm{N.}\binits{N.}} \AND
  \bauthor{\bsnm{Sanz-Serna},~\bfnm{J.~M.}\binits{J.~M.}}
(\byear{2018}).
\btitle{Geometric Integrators and the Hamiltonian {M}onte {C}arlo Method}.
\bjournal{Acta Numerica}
\bvolume{27}
\bpages{113--206}.
\end{barticle}
\endbibitem

\bibitem{BouRabeeSchuh2023}
\begin{barticle}[author]
\bauthor{\bsnm{Bou-Rabee},~\bfnm{Nawaf}\binits{N.}} \AND
  \bauthor{\bsnm{Schuh},~\bfnm{Katharina}\binits{K.}}
(\byear{2023}).
\btitle{{Convergence of unadjusted Hamiltonian Monte Carlo for mean-field
  models}}.
\bjournal{Electronic Journal of Probability}
\bvolume{28}
\bpages{1 -- 40}.
\end{barticle}
\endbibitem

\bibitem{CaLeSt2007}
\begin{barticle}[author]
\bauthor{\bsnm{Canc\'{e}s},~\bfnm{E.}\binits{E.}},
  \bauthor{\bsnm{Legoll},~\bfnm{F.}\binits{F.}} \AND
  \bauthor{\bsnm{Stoltz},~\bfnm{G.}\binits{G.}}
(\byear{2007}).
\btitle{Theoretical and Numerical Comparison of Some Sampling Methods for
  Molecular Dynamics}.
\bjournal{Mathematical Modelling and Numerical Analysis}
\bvolume{41}
\bpages{351-389}.
\end{barticle}
\endbibitem

\bibitem{Cao_2021_IBC}
\begin{barticle}[author]
\bauthor{\bsnm{Cao},~\bfnm{Yu}\binits{Y.}},
  \bauthor{\bsnm{Lu},~\bfnm{Jianfeng}\binits{J.}} \AND
  \bauthor{\bsnm{Wang},~\bfnm{Lihan}\binits{L.}}
(\byear{2021}).
\btitle{Complexity of randomized algorithms for underdamped Langevin dynamics}.
\bjournal{Communications in Mathematical Sciences}
\bvolume{19}
\bpages{1827--1853}.
\bdoi{10.4310/cms.2021.v19.n7.a4}
\end{barticle}
\endbibitem

\bibitem{casas2022new}
\begin{barticle}[author]
\bauthor{\bsnm{Casas},~\bfnm{Fernando}\binits{F.}},
  \bauthor{\bsnm{Sanz-Serna},~\bfnm{Jes\'{u}s~Mar\'{\i}a}\binits{J.~M.}} \AND
  \bauthor{\bsnm{Shaw},~\bfnm{Luke}\binits{L.}}
(\byear{2023}).
\btitle{A New Optimality Property of Strang's Splitting}.
\bjournal{SIAM Journal on Numerical Analysis}
\bvolume{61}
\bpages{1369-1385}.
\end{barticle}
\endbibitem

\bibitem{chen2020fast}
\begin{barticle}[author]
\bauthor{\bsnm{Chen},~\bfnm{Yuansi}\binits{Y.}},
  \bauthor{\bsnm{Dwivedi},~\bfnm{Raaz}\binits{R.}},
  \bauthor{\bsnm{Wainwright},~\bfnm{Martin~J}\binits{M.~J.}} \AND
  \bauthor{\bsnm{Yu},~\bfnm{Bin}\binits{B.}}
(\byear{2020}).
\btitle{Fast mixing of Metropolized Hamiltonian Monte Carlo: Benefits of
  multi-step gradients}.
\bjournal{Journal of Machine Learning Research}
\bvolume{21}
\bpages{1--72}.
\end{barticle}
\endbibitem

\bibitem{chen2019optimal}
\begin{barticle}[author]
\bauthor{\bsnm{Chen},~\bfnm{Z.}\binits{Z.}} \AND
  \bauthor{\bsnm{Vempala},~\bfnm{S.~S.}\binits{S.~S.}}
(\byear{2022}).
\btitle{{Optimal convergence rate of Hamiltonian Monte Carlo for strongly
  logconcave distributions}}.
\bjournal{Theory of Computing}
\bvolume{18}
\bpages{1--18}.
\bdoi{10.4086/toc.2022.v018a009}
\end{barticle}
\endbibitem

\bibitem{cheng2018sharp}
\begin{barticle}[author]
\bauthor{\bsnm{Cheng},~\bfnm{Xiang}\binits{X.}},
  \bauthor{\bsnm{Chatterji},~\bfnm{Niladri~S}\binits{N.~S.}},
  \bauthor{\bsnm{Abbasi-Yadkori},~\bfnm{Yasin}\binits{Y.}},
  \bauthor{\bsnm{Bartlett},~\bfnm{Peter~L}\binits{P.~L.}} \AND
  \bauthor{\bsnm{Jordan},~\bfnm{Michael~I}\binits{M.~I.}}
(\byear{2018}).
\btitle{Sharp convergence rates for Langevin dynamics in the nonconvex
  setting}.
\bjournal{arXiv preprint arXiv:1805.01648}.
\end{barticle}
\endbibitem

\bibitem{cheng2018underdamped}
\begin{binproceedings}[author]
\bauthor{\bsnm{Cheng},~\bfnm{X.}\binits{X.}},
  \bauthor{\bsnm{Chatterji},~\bfnm{N.~S.}\binits{N.~S.}},
  \bauthor{\bsnm{Bartlett},~\bfnm{P.~L.}\binits{P.~L.}} \AND
  \bauthor{\bsnm{Jordan},~\bfnm{M.~I.}\binits{M.~I.}}
(\byear{2018}).
\btitle{{Underdamped Langevin MCMC: A non-asymptotic analysis}}.
In \bbooktitle{Conference On Learning Theory}
\bpages{300--323}.
\end{binproceedings}
\endbibitem

\bibitem{dalalyan2020sampling}
\begin{barticle}[author]
\bauthor{\bsnm{Dalalyan},~\bfnm{Arnak~S}\binits{A.~S.}} \AND
  \bauthor{\bsnm{Riou-Durand},~\bfnm{Lionel}\binits{L.}}
(\byear{2020}).
\btitle{On sampling from a log-concave density using kinetic Langevin
  diffusions}.
\bjournal{Bernoulli}
\bvolume{26}
\bpages{1956--1988}.
\end{barticle}
\endbibitem

\bibitem{Daun2011}
\begin{barticle}[author]
\bauthor{\bsnm{Daun},~\bfnm{Thomas}\binits{T.}}
(\byear{2011}).
\btitle{On the randomized solution of initial value problems}.
\bjournal{Journal of Complexity}
\bvolume{27}
\bpages{300-311}.
\bnote{Dagstuhl 2009}.
\bdoi{https://doi.org/10.1016/j.jco.2010.07.002}
\end{barticle}
\endbibitem

\bibitem{DuKePeRo1987}
\begin{barticle}[author]
\bauthor{\bsnm{Duane},~\bfnm{S.}\binits{S.}},
  \bauthor{\bsnm{Kennedy},~\bfnm{A.~D.}\binits{A.~D.}},
  \bauthor{\bsnm{Pendleton},~\bfnm{B.~J.}\binits{B.~J.}} \AND
  \bauthor{\bsnm{Roweth},~\bfnm{D.}\binits{D.}}
(\byear{1987}).
\btitle{Hybrid {M}onte-{C}arlo}.
\bjournal{Phys Lett B}
\bvolume{195}
\bpages{216--222}.
\end{barticle}
\endbibitem

\bibitem{DurmusEberle2024}
\begin{barticle}[author]
\bauthor{\bsnm{Durmus},~\bfnm{Alain}\binits{A.}} \AND
  \bauthor{\bsnm{Eberle},~\bfnm{Andreas}\binits{A.}}
(\byear{2024}).
\btitle{Asymptotic bias of inexact Markov Chain Monte Carlo methods in high
  dimension}.
\bjournal{The Annals of Applied Probability}
\bvolume{34}
\bpages{3435--3468}.
\end{barticle}
\endbibitem

\bibitem{durmus2019high}
\begin{barticle}[author]
\bauthor{\bsnm{Durmus},~\bfnm{Alain}\binits{A.}} \AND
  \bauthor{\bsnm{Moulines},~\bfnm{Eric}\binits{E.}}
(\byear{2019}).
\btitle{High-dimensional Bayesian inference via the unadjusted Langevin
  algorithm}.
\bjournal{Bernoulli}
\bvolume{25}
\bpages{2854--2882}.
\end{barticle}
\endbibitem

\bibitem{FaSaSk2014}
\begin{barticle}[author]
\bauthor{\bsnm{Fang},~\bfnm{Y.}\binits{Y.}},
  \bauthor{\bsnm{Sanz-Serna},~\bfnm{J.~M.}\binits{J.~M.}} \AND
  \bauthor{\bsnm{Skeel},~\bfnm{R.~D.}\binits{R.~D.}}
(\byear{2014}).
\btitle{Compressible generalized hybrid {M}onte {C}arlo}.
\bjournal{Journal of chemical physics}
\bvolume{140}
\bpages{174108}.
\end{barticle}
\endbibitem

\bibitem{FTOshiftedode2021}
\begin{barticle}[author]
\bauthor{\bsnm{Foster},~\bfnm{James}\binits{J.}},
  \bauthor{\bsnm{Lyons},~\bfnm{Terry}\binits{T.}} \AND
  \bauthor{\bsnm{Oberhauser},~\bfnm{Harald}\binits{H.}}
(\byear{2021}).
\btitle{The shifted ODE method for underdamped Langevin MCMC}.
\bdoi{10.48550/ARXIV.2101.03446}
\end{barticle}
\endbibitem

\bibitem{ErgodicityRMMHYB}
\begin{binproceedings}[author]
\bauthor{\bsnm{He},~\bfnm{Ye}\binits{Y.}},
  \bauthor{\bsnm{Balasubramanian},~\bfnm{Krishnakumar}\binits{K.}} \AND
  \bauthor{\bsnm{Erdogdu},~\bfnm{Murat~A}\binits{M.~A.}}
(\byear{2020}).
\btitle{On the Ergodicity, Bias and Asymptotic Normality of Randomized Midpoint
  Sampling Method}.
In \bbooktitle{Advances in Neural Information Processing Systems}
(\beditor{\bfnm{H.}\binits{H.}~\bsnm{Larochelle}},
  \beditor{\bfnm{M.}\binits{M.}~\bsnm{Ranzato}},
  \beditor{\bfnm{R.}\binits{R.}~\bsnm{Hadsell}},
  \beditor{\bfnm{M.~F.}\binits{M.~F.}~\bsnm{Balcan}} \AND
  \beditor{\bfnm{H.}\binits{H.}~\bsnm{Lin}}, eds.)
\bvolume{33}
\bpages{7366--7376}.
\bpublisher{Curran Associates, Inc.}
\end{binproceedings}
\endbibitem

\bibitem{HoGe2014}
\begin{barticle}[author]
\bauthor{\bsnm{Hoffman},~\bfnm{M.~D.}\binits{M.~D.}} \AND
  \bauthor{\bsnm{Gelman},~\bfnm{A.}\binits{A.}}
(\byear{2014}).
\btitle{The no-U-turn sampler: Adaptively setting path lengths in Hamiltonian
  Monte Carlo}.
\bjournal{Journal of Machine Learning Research}
\bvolume{15}
\bpages{1593--1623}.
\end{barticle}
\endbibitem

\bibitem{hoffman2022tuning}
\begin{binproceedings}[author]
\bauthor{\bsnm{Hoffman},~\bfnm{Matthew~D}\binits{M.~D.}} \AND
  \bauthor{\bsnm{Sountsov},~\bfnm{Pavel}\binits{P.}}
(\byear{2022}).
\btitle{Tuning-Free Generalized Hamiltonian Monte Carlo}.
In \bbooktitle{International Conference on Artificial Intelligence and
  Statistics}
\bpages{7799--7813}.
\bpublisher{PMLR}.
\end{binproceedings}
\endbibitem

\bibitem{Kacewicz2004}
\begin{barticle}[author]
\bauthor{\bsnm{Kacewicz},~\bfnm{Boleslaw}\binits{B.}}
(\byear{2004}).
\btitle{Randomized and quantum algorithms yield a speed-up for initial-value
  problems}.
\bjournal{Journal of Complexity}
\bvolume{20}
\bpages{821-834}.
\bdoi{https://doi.org/10.1016/j.jco.2004.05.002}
\end{barticle}
\endbibitem

\bibitem{Kacewicz2005}
\begin{barticle}[author]
\bauthor{\bsnm{Kacewicz},~\bfnm{Boleslaw}\binits{B.}}
(\byear{2005}).
\btitle{Almost Optimal Solution of Initial-Value Problems by Randomized and
  Quantum Algorithms}.
\bdoi{10.48550/ARXIV.QUANT-PH/0510045}
\end{barticle}
\endbibitem

\bibitem{kleppe2022}
\begin{barticle}[author]
\bauthor{\bsnm{Kleppe},~\bfnm{Tore~Selland}\binits{T.~S.}}
(\byear{2022}).
\btitle{Connecting the Dots: Numerical Randomized Hamiltonian Monte Carlo with
  State-Dependent Event Rates}.
\bjournal{Journal of Computational and Graphical Statistics}
\bvolume{31}
\bpages{1238-1253}.
\bdoi{10.1080/10618600.2022.2066679}
\end{barticle}
\endbibitem

\bibitem{lee2018algorithmic}
\begin{barticle}[author]
\bauthor{\bsnm{Lee},~\bfnm{Yin~Tat}\binits{Y.~T.}},
  \bauthor{\bsnm{Song},~\bfnm{Zhao}\binits{Z.}} \AND
  \bauthor{\bsnm{Vempala},~\bfnm{Santosh~S}\binits{S.~S.}}
(\byear{2018}).
\btitle{Algorithmic theory of ODEs and sampling from well-conditioned
  logconcave densities}.
\bjournal{arXiv preprint arXiv:1812.06243}.
\end{barticle}
\endbibitem

\bibitem{lu2022explicit}
\begin{barticle}[author]
\bauthor{\bsnm{Lu},~\bfnm{Jianfeng}\binits{J.}} \AND
  \bauthor{\bsnm{Wang},~\bfnm{Lihan}\binits{L.}}
(\byear{2022}).
\btitle{On explicit L 2-convergence rate estimate for piecewise deterministic
  Markov processes in MCMC algorithms}.
\bjournal{Ann. Appl. Probab.}
\bvolume{32}
\bpages{1333--1361}.
\end{barticle}
\endbibitem

\bibitem{Ma1989}
\begin{barticle}[author]
\bauthor{\bsnm{Mackenzie},~\bfnm{P.~B.}\binits{P.~B.}}
(\byear{1989}).
\btitle{An improved hybrid {M}onte {C}arlo method}.
\bjournal{Physics Letters B}
\bvolume{226}
\bpages{369--371}.
\end{barticle}
\endbibitem

\bibitem{mangoubi2017rapid}
\begin{barticle}[author]
\bauthor{\bsnm{Mangoubi},~\bfnm{Oren}\binits{O.}} \AND
  \bauthor{\bsnm{Smith},~\bfnm{Aaron}\binits{A.}}
(\byear{2021}).
\btitle{{Mixing of Hamiltonian Monte Carlo on strongly log-concave
  distributions: Continous Dynamics}}.
\bjournal{Ann. Appl. Probab.}
\bvolume{31}
\bpages{2019-2045}.
\end{barticle}
\endbibitem

\bibitem{MaStHi2002}
\begin{barticle}[author]
\bauthor{\bsnm{Mattingly},~\bfnm{J.~C.}\binits{J.~C.}},
  \bauthor{\bsnm{Stuart},~\bfnm{A.~M.}\binits{A.~M.}} \AND
  \bauthor{\bsnm{Higham},~\bfnm{D.~J.}\binits{D.~J.}}
(\byear{2002}).
\btitle{Ergodicity for {SDE}s and approximations: locally {L}ipschitz vector
  fields and degenerate noise}.
\bjournal{Stoch. Proc. Appl.}
\bvolume{101}
\bpages{185--232}.
\end{barticle}
\endbibitem

\bibitem{mattingly2010convergence}
\begin{barticle}[author]
\bauthor{\bsnm{Mattingly},~\bfnm{J.~C.}\binits{J.~C.}},
  \bauthor{\bsnm{Stuart},~\bfnm{A.~M.}\binits{A.~M.}} \AND
  \bauthor{\bsnm{Tretyakov},~\bfnm{M.~V.}\binits{M.~V.}}
(\byear{2010}).
\btitle{Convergence of numerical time-averaging and stationary measures via
  {P}oisson equations}.
\bjournal{SIAM J Num Anal}
\bvolume{48}
\bpages{552--577}.
\end{barticle}
\endbibitem

\bibitem{Milstein2021}
\begin{binbook}[author]
\bauthor{\bsnm{Milstein},~\bfnm{Grigori~N.}\binits{G.~N.}} \AND
  \bauthor{\bsnm{Tretyakov},~\bfnm{Michael~V.}\binits{M.~V.}}
(\byear{2021}).
\btitle{Mean-Square Approximation for Stochastic Differential Equations}
In \bbooktitle{Stochastic Numerics for Mathematical Physics}
\bpages{1--94}.
\bpublisher{Springer International Publishing}, \baddress{Cham}.
\end{binbook}
\endbibitem

\bibitem{monmarche2022united}
\begin{barticle}[author]
\bauthor{\bsnm{Monmarch{\'e}},~\bfnm{Pierre}\binits{P.}}
(\byear{2022}).
\btitle{HMC and underdamped Langevin united in the unadjusted convex smooth
  case}.
\bdoi{10.48550/ARXIV.2202.00977}
\end{barticle}
\endbibitem

\bibitem{Ne2011}
\begin{barticle}[author]
\bauthor{\bsnm{Neal},~\bfnm{R.~M.}\binits{R.~M.}}
(\byear{2011}).
\btitle{{MCMC} using {H}amiltonian dynamics}.
\bjournal{Handbook of {M}arkov {C}hain {M}onte {C}arlo}
\bvolume{2}
\bpages{113-162}.
\end{barticle}
\endbibitem

\bibitem{quispel2008new}
\begin{barticle}[author]
\bauthor{\bsnm{Quispel},~\bfnm{GRW}\binits{G.}} \AND
  \bauthor{\bsnm{McLaren},~\bfnm{David~Ian}\binits{D.~I.}}
(\byear{2008}).
\btitle{A new class of energy-preserving numerical integration methods}.
\bjournal{Journal of Physics A: Mathematical and Theoretical}
\bvolume{41}
\bpages{045206}.
\end{barticle}
\endbibitem

\bibitem{SaSt1999}
\begin{barticle}[author]
\bauthor{\bsnm{Sanz-Serna},~\bfnm{J.~M.}\binits{J.~M.}} \AND
  \bauthor{\bsnm{Stuart},~\bfnm{A.~M.}\binits{A.~M.}}
(\byear{1999}).
\btitle{Ergodicity of dissipative differential equations subject to random
  impulses}.
\bjournal{Journal of Differential Equations}
\bvolume{155}
\bpages{262--284}.
\end{barticle}
\endbibitem

\bibitem{shen2019randomized}
\begin{barticle}[author]
\bauthor{\bsnm{Shen},~\bfnm{Ruoqi}\binits{R.}} \AND
  \bauthor{\bsnm{Lee},~\bfnm{Yin~Tat}\binits{Y.~T.}}
(\byear{2019}).
\btitle{The randomized midpoint method for log-concave sampling}.
\bjournal{Advances in Neural Information Processing Systems}
\bvolume{32}.
\end{barticle}
\endbibitem

\bibitem{Ta2002}
\begin{barticle}[author]
\bauthor{\bsnm{Talay},~\bfnm{D.}\binits{D.}}
(\byear{2002}).
\btitle{Stochastic {H}amiltonian Systems: Exponential Convergence to the
  Invariant Measure, and Discretization by the Implicit {E}uler Scheme}.
\bjournal{Markov Processes and Related Fields}
\bvolume{8}
\bpages{1--36}.
\end{barticle}
\endbibitem

\bibitem{vogrinc2021counterexamples}
\begin{barticle}[author]
\bauthor{\bsnm{Vogrinc},~\bfnm{Jure}\binits{J.}} \AND
  \bauthor{\bsnm{Kendall},~\bfnm{Wilfrid~S}\binits{W.~S.}}
(\byear{2021}).
\btitle{Counterexamples for optimal scaling of Metropolis--Hastings chains with
  rough target densities}.
\bjournal{Ann. Appl. Probab.}
\bvolume{31}
\bpages{972--1019}.
\end{barticle}
\endbibitem

\bibitem{wu2001large}
\begin{barticle}[author]
\bauthor{\bsnm{Wu},~\bfnm{Liming}\binits{L.}}
(\byear{2001}).
\btitle{Large and moderate deviations and exponential convergence for
  stochastic damping Hamiltonian systems}.
\bjournal{Stochastic processes and their applications}
\bvolume{91}
\bpages{205--238}.
\end{barticle}
\endbibitem

\end{thebibliography}

\end{document}